\documentclass{amsart}

\usepackage{slashed}
\usepackage{hyperref}	
\usepackage{adjustbox}
\usepackage{amsmath}	
\usepackage{amssymb}
\usepackage{graphicx}
\usepackage{mathrsfs}	
\usepackage{multicol}
\usepackage{soul}	
\usepackage{color}	
\usepackage{appendix}	
\usepackage{eufrak}
\usepackage[labelformat=simple]{subcaption}
\usepackage[backend=bibtex,style=numeric,citestyle=numeric-comp,maxbibnames=99]{biblatex}
\renewbibmacro{in:}{\ifentrytype{article}{}{\printtext{\bibstring{in}\intitlepunct}}}
 
\newcommand{\de}{\ensuremath{\partial}}
\newcommand{\dee}{\ensuremath{\textrm{d}}}

\newcommand{\inty}[4]{\ensuremath{ \int_{#1}^{#2} \! #3 \, \dee#4 }}

\newcommand{\field}[1]{\mathbb{#1}}
\newcommand{\ip}[2]{\ensuremath{ \left< \left. #1 \right| #2 \right> } }

\newcommand{\Prl}{{P^\perp_{_{\rm RL}}}} 
\newcommand{\Hrl}{\mathcal{H}^\perp_{_{\rm RL}}}
\newcommand{\Prol}{{P^\perp_{_{\rm R0L}}}} 
\newcommand{\Hrol}{\mathcal{H}^\perp_{_{\rm R0L}}}

\renewcommand{\d}{\delta}
\newcommand{\+}{R}
\renewcommand{\-}{L}

\DeclareMathOperator{\sech}{sech}

\newtheorem{definition}{Definition}
\newtheorem{assumption}{Assumption}
\newtheorem{remark}{Remark}
\newtheorem{theorem}{Theorem}
\newtheorem{lemma}{Lemma}
\newtheorem{proposition}{Proposition}
\newtheorem{corollary}[proposition]{Corollary}
\newtheorem{example}{Example}
\let\oldhat\hat
\renewcommand{\hat}[1]{\oldhat{\boldsymbol{#1}}}
\bibliography{bibliography.bib}

\numberwithin{lemma}{section}
\numberwithin{example}{section}
\numberwithin{figure}{section}
\numberwithin{proposition}{section}
\numberwithin{equation}{section}
\numberwithin{theorem}{section}
\numberwithin{remark}{section}
\numberwithin{definition}{section}
\numberwithin{assumption}{section}

\allowdisplaybreaks

\title{Dirac operators and domain walls}

\author{Jianfeng Lu}
\address{Department of Mathematics, Department of Physics, and Department of Chemistry, Duke University, Box 90320, Durham, NC 27708}\email{jianfeng@math.duke.edu}
\author{Alexander B. Watson}
\address{Department of Mathematics, Duke University, Box 90320, Durham, NC 27708}
\email{abwatson@math.duke.edu}
\author{Michael I. Weinstein}
\address{ Department of Applied Physics and Applied Mathematics,               Department of Mathematics, Columbia University, New York, NY 10025}
\email{miw2103@columbia.edu}

\begin{document}

\maketitle

\begin{abstract} 
We study the eigenvalue problem for a one-dimensional Dirac operator with a spatially varying ``mass'' term. It is well-known that when the mass function has the form of a kink, or \emph{domain wall}, transitioning between strictly  positive and strictly negative asymptotic mass, $\pm\kappa_\infty$, at $\pm\infty$, the Dirac operator has a simple eigenvalue of zero energy (geometric multiplicity equal to one) within a  gap in the continuous spectrum, with corresponding exponentially localized \emph{zero mode}. 

We consider the eigenvalue problem for the one-dimensional Dirac operator with mass function defined by ``glue-ing'' together $n$ domain wall-type transitions, assuming that the distance between transitions, $2 \d$, is sufficiently large, focusing on the illustrative cases $n = 2$ and $3$. When $n = 2$
we prove that the Dirac operator has two real simple eigenvalues of opposite sign and of order $e^{- 2 |\kappa_\infty| \d}$. The associated eigenfunctions are, up to $L^2$ error of order $e^{- 2 |\kappa_\infty| \d}$, linear combinations of shifted copies of the single domain wall zero mode. For the case $n = 3$, we prove the Dirac operator has two non-zero simple eigenvalues as in the two domain wall case and a simple eigenvalue at energy zero. The associated eigenfunctions of these eigenvalues can again, up to small error, be expressed as linear combinations of shifted copies of the single domain wall zero mode. When $n > 3$ no new technical difficulty arises and the result is similar. Our methods are based on a Lyapunov-Schmidt reduction/Schur complement strategy, which maps the Dirac operator eigenvalue problem for eigenstates with near-zero energies 
to the problem of determining the kernel of an $n \times n$ matrix reduction, which depends nonlinearly on the eigenvalue parameter.

The class of Dirac operators we consider controls the bifurcation of topologically protected ``edge states'' from Dirac points (linear band crossings) for classes of Schr\"odinger operators with domain-wall modulated periodic potentials in one and two space dimensions. The present results may be used to construct a rich class of defect modes in periodic structures modulated by multiple domain walls. 
\end{abstract} 
\tableofcontents

\section{Introduction} 

It is well-known that spatially localized defects in crystalline 
media, described by the Schr\"odinger equation with a periodic
potential, may give rise to bound states localized to the defect.  The underlying periodic differential operator
has spectrum which is continuous  and
a spatially localized perturbation, which breaks translation
invariance, induces the bifurcation of discrete eigenvalues (with corresponding defect eigenstates) from a
continuous spectral band edge into a spectral gap of the unperturbed operator. See, for example, \cite{Deift-Hempel:86,Figotin-Klein:97,Figotin-Klein:98,BS:03,BS:06,Borisov-Gadylshin:08,parzygnat2010sufficient,Borisov11,Hoefer-Weinstein:11,DVW:14,DVW:15,Drouot1,Drouot2,Drouot3}. This paper is related to a mechanism for the emergence of defect states due to perturbations which are not spatially localized.
Such perturbations may arise in models of dislocations in crystals; see, for example, \cite{Korotyaev:00,Hempel-Kohlmann:11}.

The present work relates to a class of non-compact line defect ({\it
  edge}) perturbations. These are defined via a domain wall
interpolation between ``gapped'' asymptotic periodic structures
(deformed two-dimensional honeycomb structures), studied in
\cite{FLW-APDE:16,FLW-2DMat:16}. Such perturbations were proved to
induce bifurcating branches of ``edge modes'' from Dirac points,
conical intersections of spectral bands of the bulk honeycomb
structure. These modes are propagating (plane-wave like) in the
direction parallel to the edge and are localized transverse to the
edge. Their transverse localization is determined by the eigenstates
of an effective one-dimensional Dirac operator  
$\mathcal{D}_\kappa=i\sigma_3\partial_x +\kappa(x)\sigma_1$,
where $\kappa(x)$ is a spatially varying mass term. The function $\kappa(x)$  enters as a  {\it domain wall function}  in
the definition of the line-defect; $\kappa(x)$ transitions between
asymptotic values of opposite sign, $\pm\kappa_\infty$, as the
distance from the edge tends to infinity on opposite sides of the
edge; see Figure \ref{fig:kappa}. For the formal derivation of the operator $\mathcal{D}_\kappa$ in this context, see Section 6 of \cite{FLW-APDE:16}, in particular (6.21-22). Analogous results were
obtained for defect modes in a class of one-dimensional modulated
periodic structures in \cite{FLW-PNAS:14,FLW-Memoirs:17,DFW:18}. A
photonic realization of such structures is studied in
\cite{PRA_photonic:16,Fu-etal:16}. 

The operator $\mathcal{D}_\kappa$, for $\kappa$ of the above type, 
always has a zero energy eigenstate and, in general, has an odd number of eigenvalues in the
spectral gap.
The zero mode is {\it topologically
  protected}; an arbitrary localized perturbation of $\kappa$ does not destroy 
  the eigenvalue at zero energy. The persistence of this zero-energy eigenstate implies, by the above discussion, the persistence of the corresponding  family of bifurcating edge modes against arbitrary spatially localized (even large) perturbations of $\kappa$. 
  The sense in which these edge modes inherit the topologically protected character of the zero mode of $\mathcal{D}_\kappa$ is studied in \cite{Dr-top}.

 The scientific and technological interest in  topologically robust  modes is due to their  
potential as highly robust channels for transmission of energy and information; see \cite{2014LuJoannopoulosSoljacic,2017KhanikaevShvets,ozawa_etal:18} for recent reviews in the photonic setting.  \\

\noindent Our study concerns detailed properties of the Dirac operator $\mathcal{D}_\kappa$.\\

\noindent {\it Question: Suppose we ``glue together'' two or more
  domain walls (see Figure~\ref{fig:kappa_L} and
  Figure~\ref{fig:3_dw_kappa}, for example). What can be said of the
  edge states of such structures? Can we construct defect modes of the
  Dirac operator which are approximated by weighted superpositions of
  spatial-translates of isolated domain wall zero-energy eigenmodes?}
   \\
 
The goal of this article is to show that such constructions are
possible provided the separation between the domain wall transitions
(cores) is sufficiently large.  In particular we show for any integer
$N\ge1$:
\begin{enumerate}
\item[(a)] For multiple separated domain walls with $2N$ transitions
  (e.g. Figure~\ref{fig:kappa_L}, $N=1$), the Dirac operator has $2N$ simple non-zero eigenvalues, which are exponentially near zero. Moreover,   zero is not an eigenvalue. 
\item[(b)] For multiple separated domain walls with $2N+1$ transitions
  (e.g. Figure~\ref{fig:3_dw_kappa}, $N=1$), the Dirac operator has
  $2N$ simple non-zero eigenvalues and a simple eigenvalue at zero.
\end{enumerate}
  
By the correspondence outlined above, between eigenmodes of the
effective Dirac operator and edge states of the perturbed
Schr\"{o}dinger operator, there exists a rich family of bound states
in two-dimensional crystalline media with multiple parallel edge
defects. Note that bound states seeded by non-zero eigenvalues of the
effective one-dimensional Dirac operator do not share the same
robustness as those seeded by exact zero modes. Such modes may be
destroyed by appropriate localized deformation of the structure; for
example the double domain-wall function displayed in Figure
\ref{fig:kappa_L} may be deformed to the constant function $1$ by
adding an appropriate localized perturbation.

For simplicity of exposition, in this work we focus mostly on the illustrative cases of two and three domain walls, although the theory we present extends naturally to the more general setting outlined above. For the simplest case of two domain walls we prove (Theorem \ref{th:main_theorem}) the following:
\begin{quote}
\it  Let $\mathcal{D}_{\kappa^\d}=i\sigma_3\partial_x\ +\ \kappa^\delta(x) \sigma_1$ denote the one-dimensional Dirac operator with mass function $\kappa^\d(x)$, having the form of two domain walls separated by a distance $2 \d$ (Figure \ref{fig:kappa_L}). Then, for sufficiently large $\d$, the operator $\mathcal{D}_{\kappa^\d}$ has two simple  non-zero eigenvalues, $E$, of order $e^{- 2 \kappa_\infty \d}$ where $\kappa_\infty > 0$ is a constant. The associated eigenfunctions of these eigenvalues are, up to error of order $e^{- 2 \kappa_\infty \d}$ in $L^2$, linear combinations of shifted copies of the single domain wall zero mode (Figure \ref{fig:two_domain_walls_modes}). 
\end{quote}
%
We remark that the discrete spectrum of $\mathcal{D}_\kappa$ is simple;
 see \cite{DFW:18}.
Theorem \ref{th:main_theorem} and its variants for three and $n > 3$ domain walls (Theorems \ref{th:3_dw} and \ref{th:n_theorem}, respectively) are proved
using a Lyapunov-Schmidt reduction/Schur complement strategy (Section
\ref{sec:strategy_of_proof}). For the case of two domain walls, the
eigenvalue problem on $L^2(\mathbb{R})$ is projected onto a natural
$2$--dimensional subspace yielding an equivalent ``nonlinear
eigenvalue problem'' on $\mathbb{C}^2$.  The original eigenvalue
problem has an eigenstate with energy in a neighborhood of zero if and
only if $\det M_{_{2\times2}}(E,\delta)=0$, where
$M_{_{2\times2}}(E,\delta)$ is a $2\times 2$ matrix, which depends
nonlinearly and analytically on the eigenvalue ($E$) and separate
parameter ($\delta$).  The two dimensional subspace is given in terms
of spatial translates of the single domain-wall zero energy
eigenstate, centered on the domain wall transitions of
$\kappa^\delta$.
That this reduction is valid for large separations between domain walls, $\delta$,  follows from the energy estimate  of Proposition \ref{prop:prop_on_DL}. The proof, given in Appendix \ref{sec:proofs_1}, uses a spatial partition of unity to express the Hamiltonian in terms of localized operators, near and away from the domain wall ``cores''. Detailed information about the Dirac operator eigenvalues and bound states then follows from a careful  expansion of $M_{_{2\times2}}(E,\delta)$  for large $\delta$ and $E$ in an appropriate compact set (Proposition  \ref{lem:bounds_on_terms_in_M}, proved in Appendix \ref{sec:proofs_2}). 

In the case of $n$ equally spaced domain walls the reduction is to a nonlinear eigenvalue problem for an $n \times n$ matrix, $M_{_{n\times n}}(E,\delta)$. We present the general result (Theorem \ref{th:n_theorem}), giving details of the proof for the case $n=3$ in Section \ref{sec:3_dom_wall_case} 
(Theorem \ref{th:3_dw}; bound states plotted in Figure \ref{fig:three_domain_walls_modes}). Our methods extend to the case of non-equally spaced arrangements of domain walls; see Remark \ref{rem:remark_on_nonequal_spacing}. 

The one-dimensional Dirac equation with spatially varying mass is closely related with the Su-Schrieffer-Heeger (SSH) model of polyacetylene \cite{1980SuSchriefferHeeger}. This operator also plays a role in \cite{1976JackiwRebbi} in the context of quantum field theory. Continuum models of polyacetylene which incorporate the effect of lattice distortions have also been developed \cite{1980TakayamaLin-LiuMaki}. Solutions of such models describing two well-spaced ``kinks'' were derived in \cite{1982CampbellBishop}. 

Our results have their counterparts in the context of Schr\"odinger operators with two or more \emph{potential wells} related by a symmetry in the semi-classical regime \cite{1980Harrell,1984Simon,1984HelfferRobert,1984HelfferSjostrand,1985HelfferSjostrand,1985KirschSimon,1986Nakamura,1986Martinez,1987KirschSimon,1988MartinezRouleux,1988GerardGrigis,Dimassi-Sjoestrand:99}. Indeed, after a change of basis, the square of the Dirac operator we consider is diagonal with entries given by such Schr\"odinger operators, more precisely with the form of ``Witten Laplacians'' \cite{1982Witten,CyconFroeseKirschSimon}. In Appendix \ref{ap:remark_on_schrodinger_link} we present the details of this connection and sketch an alternative proof of our main results which relies on the well-established semiclassical analysis of such operators. Since we are able to give more direct and elementary proofs we do not make further use of this connection in this work. The analogous problem for Dirac operators (again with potentials given by well-separated wells) was studied in 
\cite{1983HarrellKlaus,1985Wang}. 

\subsection*{Acknowledgements} 
The authors thank Alexis Drouot for stimulating discussions. 
Part of this research was done while M.I.W. was Bergman Visiting Professor at Stanford University and he wishes to express his gratitude to the Department of Mathematics for its hospitality. This work was supported in part by the U.S. National Science Foundation grants DMS-1454939 (J.L.); DMS-1412560, DMS-1620418, and  Simons Foundation Math + X Investigator Award \#376319 (M.I.W.). 

\subsection{Notation} The Hilbert space $\mathcal{H} = L^2(\field{R};\field{C}^2)$ may be concretely realized as 2-vectors of $L^2(\field{R})$ functions:
\begin{equation}
	\mathcal{H} = \left\{ f = \begin{pmatrix} f_1 \\ f_2 \end{pmatrix} : \text{for } j \in \{1,2\}, f_j \in L^2(\field{R})  \right\}.
\end{equation}
We adopt the following notation for the standard inner product and induced norm on $\mathcal{H}$:
\begin{equation} \label{eq:def_of_H}
\begin{split}
	&\ip{f}{g}_{\mathcal{H}} := \sum_{j = 1,2} \ip{f_j}{g_j}_{L^2(\field{R})},\quad \|f\|_{\mathcal{H}} =\sqrt{\ip{f}{f}_{\mathcal{H}} }.
\end{split}
\end{equation}
We also introduce Sobolev spaces $\mathcal{H}^s=H^s(\mathbb{R};\mathbb{C}^2)$, for $s\ge0$ with $\mathcal{H}^0=\mathcal{H}$. 

For complex vectors $v, w$ in $\field{C}^2$, we will write their inner product and the norm induced by this inner product as: 
\begin{equation}
	\ip{v}{w} := \sum_{j = 1,2} \overline{v}_j w_j, \quad | v | := \sqrt{ \ip{v}{v} }.
\end{equation}
With this notation in hand, a short manipulation of the definition of the $\mathcal{H}$-inner product shows that:
\begin{equation}
	\ip{ f }{ g }_\mathcal{H} = \inty{\field{R}}{}{ \ip{f(x)}{g(x)} }{x}, \quad \| f \|_\mathcal{H}^2 = \inty{\field{R}}{}{ | f(x) |^2 }{x}.
\end{equation}
Analogous remarks apply to the Sobolev spaces $\mathcal{H}^s,\ s\ge0$.

\section{Statement of results} \label{sec:statement_of_results}
We now present our results in more detail, focusing first on the simplest case of two domain walls. We define $\mathcal{D}_{\kappa^\d}$ to be the Dirac operator in one spatial dimension with matrix potential $\kappa^\d(x)\sigma_1$, where $\kappa^\d(x)$ denotes a spatially varying mass depending on a parameter $\d$: 
\begin{equation} \label{eq:def_dirac}
	\mathcal{D}_{\kappa^\d} := i \de_x \sigma_3 + \kappa^\d(x) \sigma_1.
\end{equation}
Here, $\sigma_1$ and $ \sigma_3$ denote the usual Pauli matrices: 
\begin{equation}
	\sigma_1 := \begin{pmatrix} 0 & 1 \\ 1 & 0 \end{pmatrix}, \quad \sigma_3 := \begin{pmatrix} 1 & 0 \\ 0 & -1 \end{pmatrix}.
\end{equation}
We consider the eigenvalue problem: 
\begin{equation} \label{eq:eigenvalue_problem}
	\mathcal{D}_{\kappa^\d} \alpha = E \alpha,\quad \alpha \in \mathcal{H} := L^2(\field{R};\field{C}^2).
\end{equation}

Let $\kappa_\infty$ denote a fixed positive constant and assume $\d>1$ (to be chosen sufficiently large at a later stage).  
We define what we mean by a {\it domain wall mass function}  or simply {\it domain wall}:
\begin{definition}[Domain wall]\label{s-dw}
We call $\kappa$ a \emph{domain wall function} or \emph{one domain wall function} if 
\begin{enumerate}
\item $\kappa\in C^1(\mathbb{R})$
\item  $\kappa$ is  odd (i.e. $\kappa(-x) = - \kappa(x)$), and 
\item $\kappa$ is piecewise constant except on a compact set,
  $[-1, 1]$:
\begin{equation} \label{eq:properties_of_kappa}
	\kappa(x) = \begin{cases} - \kappa_\infty &\text{if } x \leq -1; \\ \kappa_\infty &\text{if } x \geq 1. \end{cases}
\end{equation}
\end{enumerate}
If $\kappa$ is in addition monotonic, we call $\kappa$ a \emph{monotonic domain wall function}. 
%
\end{definition}
Definition \ref{s-dw}, in particular the piecewise constant assumption (3) along with the condition that $\kappa_\infty \neq 0$, ensures that the single domain wall operator \eqref{eq:single_dom_wall_operator} models states at a transition between two distinct structures.
%
%
%
\begin{remark}
The class of $\kappa^\d(x)$ we study in this paper are piecewise equal to a one domain wall function $\kappa$ (or its reflection) and such that the  $\mathcal{D}_\kappa$ \eqref{eq:single_dom_wall_operator}  satisfies a spectral gap assumption (Assumption \ref{as:assumption_on_D}). We will usually have in mind a monotonic domain wall function, as in Figures \ref{fig:kappa}, \ref{fig:kappa_L}, and \ref{fig:3_dw_kappa}, but monotonicity is not necessary for our proofs. See also Remarks \ref{rem:weak_dw} and \ref{rem:remark_on_weaker_assumptions} on relaxations of Definition \ref{s-dw} and Assumption \ref{as:assumption_on_D} under which our results still hold.
\end{remark} 

\begin{remark} \label{rem:weak_dw}
  Assuming that $\kappa$ is odd and varies only on a compact set,
  which we can take to be $[-1,1]$, simplifies the proof and
  statements of results.  Our methods extend to the case where
  $\kappa$ satisfies the weaker assumption that:
\begin{equation} \label{eq:weaker_1}
	\lim_{x \rightarrow \infty} \kappa(x) = \kappa_\infty \quad \lim_{x \rightarrow - \infty} \kappa(x) = - \kappa_\infty
\end{equation}
and approaches these limits sufficiently rapidly. For example, if the integrals:
\begin{equation} \label{eq:weaker_2}
	\inty{0}{\infty}{ | \kappa(y) - \kappa_\infty | }{y} \quad \inty{-\infty}{0}{ | \kappa(y) + \kappa_\infty | }{y} 
\end{equation}
are finite. Our methods extend also to the case where $\kappa(+\infty)\kappa(-\infty)<0$ and $|\kappa(+\infty)|\ne|\kappa(-\infty)|$. An example of a function satisfying properties \eqref{eq:weaker_1}-\eqref{eq:weaker_2} is $\kappa(x) = \tanh(x)$. Our results extend also to the case where $\kappa$ is less regular than $C^1$, for example when $\kappa$ is the signum function:
\begin{equation}
	\text{\emph{sgn}}(x) := \begin{cases} 1 & x \geq 0 \\ -1 & x < 0 \end{cases}
\end{equation}
which is merely in $L^\infty$. We make use of regularity of $\kappa$ in order to give an elementary proof of Corollary \ref{F0} which requires an elliptic regularity estimate for the operator $D_{\kappa^\delta}$. When $\kappa \in L^\infty$ the result still holds, but the proof requires more machinery:
see the proof of Theorem 5.6 in \cite{HislopSigal} for example.
\end{remark}

We start by summarizing some general properties satisfied by Dirac operators \eqref{eq:def_dirac} with mass functions which converge sufficiently fast to non-zero constants at $\pm \infty$, such as the operator \eqref{eq:kappa_L}, and more generally the $n$ domain wall operators we consider below.
\begin{theorem} \label{th:on_dirac}
Consider the Dirac operator:
\begin{equation}
    \mathcal{D}_{\kappa} = i \de_x \sigma_3 + \kappa(x) \sigma_1
\end{equation}
where the real variable mass function $\kappa(x)$ satisfies: 
\begin{equation} \label{eq:lim_x}
    \lim_{|x| \rightarrow \infty} \kappa^2(x) = \kappa^2_\infty
\end{equation}
for some constant $\kappa_\infty > 0$, and that the limit in \eqref{eq:lim_x} is approached sufficiently fast that the integral:
\begin{equation} \label{eq:fast_conv}
	\inty{0}{\infty}{ | \kappa^2(y) - \kappa^2_\infty | }{y} 
\end{equation}
converges. Then $\mathcal{D}_{\kappa}$ has the following properties:
\begin{enumerate}
\item $\mathcal{D}_{\kappa}$, with domain $\mathcal{H}^1 = H^1(\field{R};\field{C}^2)$, is self-adjoint with respect to the inner product on $\mathcal{H}$.
\item The spectrum of $\mathcal{D}_{\kappa}$ is symmetric about $0$.
\item The essential spectrum of $\mathcal{D}_{\kappa}$ is $(-\infty,-\kappa_\infty] \cup [\kappa_\infty,\infty)$.
\item $\mathcal{D}_{\kappa}$ has finitely many eigenvalues in the gap, $(- \kappa_\infty,\kappa_\infty)$, in  essential spectrum .
\item Eigenvalues in the gap between essential spectrum are simple.
\item If $0$ is an eigenvalue of $\mathcal{D}_{\kappa}$, then there are an odd number of eigenvalues in the gap between essential spectrum. If not, there are an even number of eigenvalues in the gap between essential spectrum.
\item $\mathcal{D}_{\kappa}$ has a simple eigenvalue at energy $E=0$ if and only if $\kappa(\infty) \kappa(- \infty) < 0$. When $\lim_{x \rightarrow \infty} \kappa(x) = \kappa_\infty > 0$, the corresponding $0$-eigenspace is spanned by the ($\mathcal{H}$-normalized) function: 
\begin{equation} \label{eq:zero_mode}
	\alpha_\star(x) := \gamma \begin{pmatrix} 1 \\ i \end{pmatrix} e^{- \inty{0}{x}{ \kappa(y) }{y} },\qquad
	 \gamma := \frac{1}{ \sqrt{2} \| e^{- \inty{0}{x}{ \kappa(y) }{y} } \|_{L^2} }.
\end{equation}
The choice of $\alpha_\star$ is unique up to multiplication by a complex constant of modulus one. By (\ref{eq:properties_of_kappa}), there exists a constant $C > 0$ depending only on $\kappa$ such that:
\begin{equation} \label{eq:alpha_for_x_large}
	| \alpha_\star(x) | \leq C e^{- \kappa_\infty |x|}, \quad | \de_x \alpha_\star(x) | \leq C e^{- \kappa_\infty |x|}.
\end{equation}
When $\lim_{x \rightarrow \infty} \kappa^\d(x) = - \kappa_\infty < 0$, then the $0$-eigenspace is spanned by
\begin{equation} \label{eq:zero_mode_flip}
	\alpha_\star(x) := \gamma \begin{pmatrix} 1 \\ - i \end{pmatrix} e^{\inty{0}{x}{ \kappa(y) }{y} },\qquad
        \gamma := \frac{1}{ \sqrt{2} \| e^{\inty{0}{x}{ \kappa(y) }{y} } \|_{L^2} }.
\end{equation}
\end{enumerate}
\end{theorem}
\begin{proof}
For (1), see, for example \cite{reed_simon_4}. (2) follows from the observation that $\mathcal{D}_{\kappa^\d}$ anti-commutes with the Pauli matrix $\sigma_2$:
\begin{equation}
    \mathcal{D}_{\kappa^\d} \sigma_2 = - \sigma_2 \mathcal{D}_{\kappa^\d}, \quad \sigma_2 = \begin{pmatrix} 0 & - i \\ i & 0 \end{pmatrix}.
\end{equation}
(3) follows from Weyl's criterion. To see (4), note that the squares of eigenvalues of $\mathcal{D}_{\kappa^\d}$ are eigenvalues of the squared Dirac operator $U \mathcal{D}_{\kappa^\d}^2 U^*$ displayed in \eqref{eq:twoo}, where $U$ is the unitary matrix \eqref{eq:U}. Assumption \eqref{eq:fast_conv} amounts to an estimate on the rate of decay of the potentials of the Schr\"odinger operators on the diagonal of \eqref{eq:twoo}. (4) then follows by standard arguments (see \cite{reed_simon_4}, for example) for Schr\"odinger operators. For (5), see Theorem 4 of Appendix C of \cite{DFW:18}. (6) follows from (2), (4), and (5). (7) can be verified by direct calculation. 
\end{proof}

We now give an explicit construction of a monotonic domain wall function.
\begin{example} \label{ex:sdw_f}
Let 
\begin{equation}
	\nu(\xi) := \begin{cases} 0 & \xi \leq 0 \\ e^{-1/\xi} & \xi > 0 \end{cases}.
\end{equation}
Note that $\nu(\xi)$ is a smooth, monotone function which approaches 1 as $\xi \rightarrow \infty$. Then the function: 
\begin{equation} \label{eq:explicit_dw_func}
	\kappa(x) = 2 \left( \frac{ \nu \left( \frac{ x + 1 }{ 2 } \right) }{ \nu \left( \frac{ x + 1 }{ 2 } \right) + 
	\nu \left( 1 - \frac{x + 1}{2} \right) } \right) - 1
\end{equation}
is a monotone function satisfying conditions (1)-(3) of Definition \ref{s-dw} with $\kappa_\infty = 1$. This function is plotted in Figure \ref{fig:kappa}. 
\end{example}

Now let $\kappa(x)$ denote a one domain wall function and define, for any $\delta>1$, the ``two domain wall'' mass function, $\kappa^\d$ as follows:
\begin{equation} \label{eq:kappa_L}
	\kappa^\d(x) = \begin{cases} - \kappa(x + \d) &\text{ for } - \infty \leq x \leq 0  \\ \kappa(x - \d) &\text{ for } 0 \leq x \leq \infty \end{cases}
\end{equation}  
See Figure \ref{fig:kappa_L}. Since we have chosen $\d > 1$ it follows that $\kappa^\d \in C^1$. By assumption \eqref{eq:properties_of_kappa} on $\kappa(x)$ we have that $(\kappa^\d(x))^2 - \kappa_\infty^2$ is supported on a compact set and hence \eqref{eq:kappa_L} satisfies the hypotheses of Theorem \ref{th:on_dirac}.


\begin{figure} 
\begin{subfigure}[b]{.45\textwidth}
\includegraphics[scale=.35]{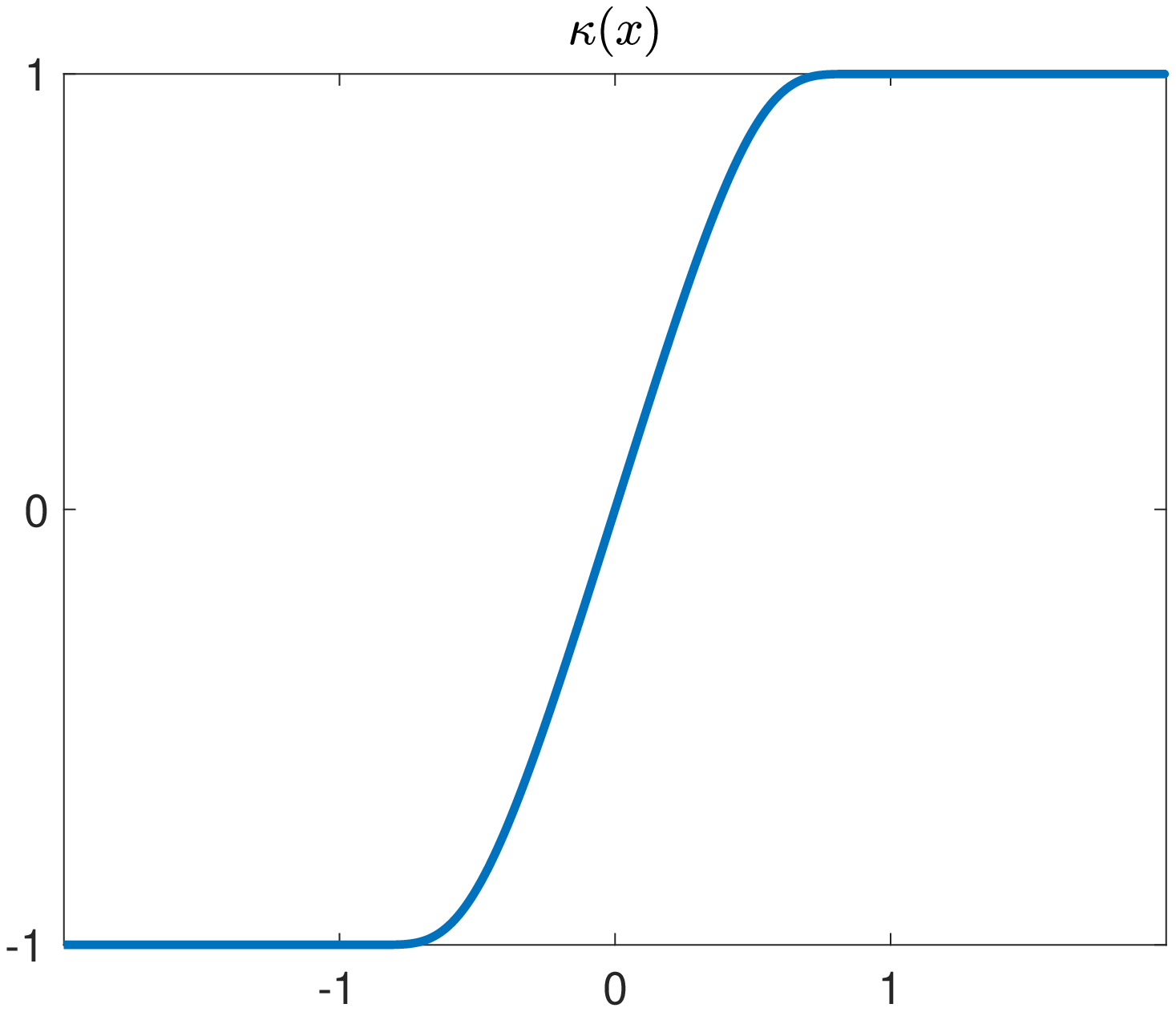}
\caption{}
\label{fig:kappa}
\end{subfigure}
\begin{subfigure}[b]{.45\textwidth}
\includegraphics[scale=.35]{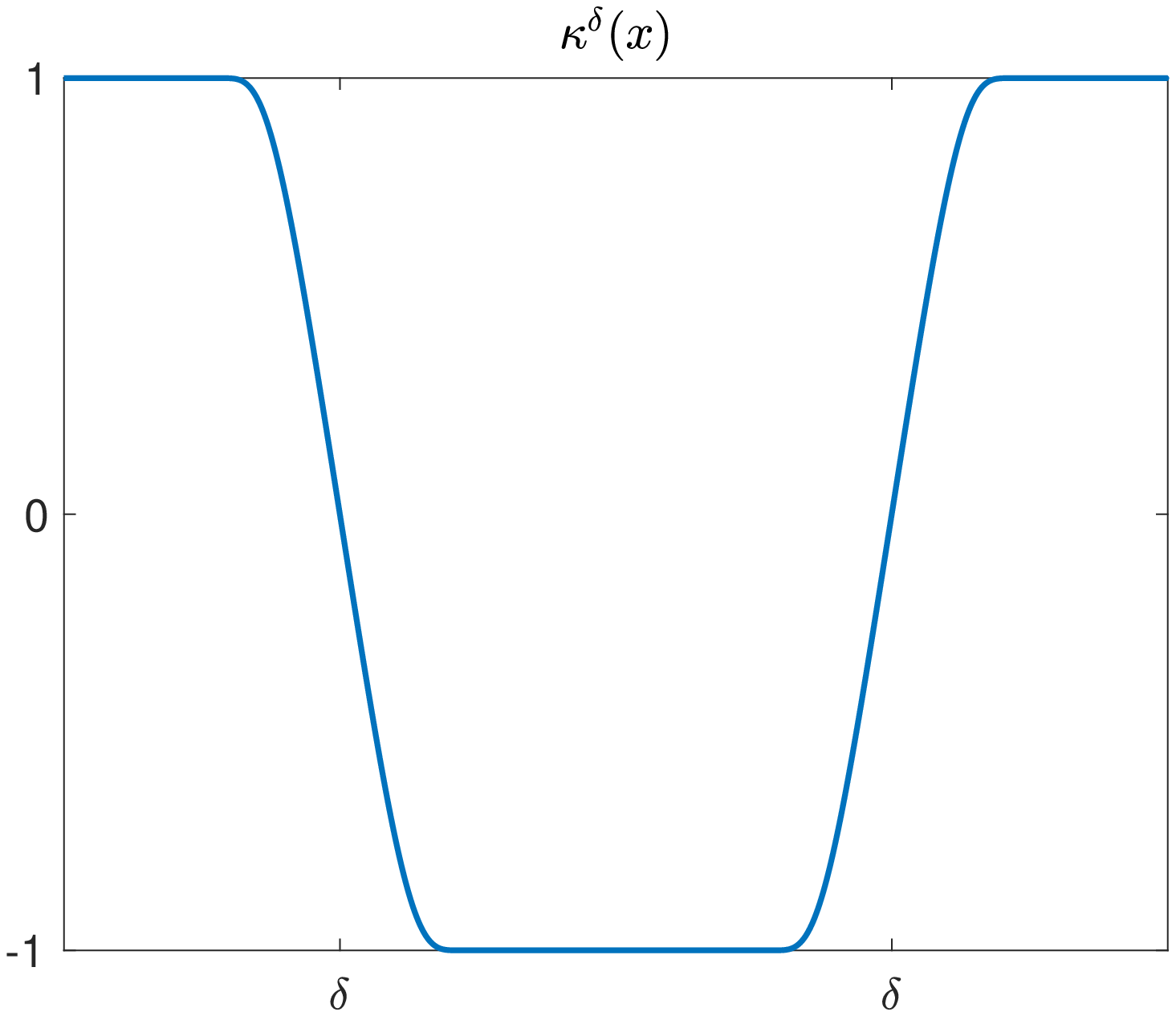} 
\caption{}
\label{fig:kappa_L}
\end{subfigure}
\caption{(a) Plot of the one domain wall function $\kappa(x)$ defined by \eqref{eq:explicit_dw_func}, and (b) plot of the two domain wall function $\kappa^\d(x)$ defined by \eqref{eq:kappa_L} with $\kappa(x)$ given by \eqref{eq:explicit_dw_func} and $\d = 2$.}
\end{figure}

\subsection{One domain wall Dirac operator $\mathcal{D}_{\kappa}$ and spectral gap assumption} Let $\mathcal{D}_{\kappa}$ denote the ``one domain wall'' Dirac operator:
\begin{equation} \label{eq:single_dom_wall_operator}
	\mathcal{D}_\kappa := i \sigma_3 \de_x  + \kappa(x) \sigma_1,
\end{equation}
where $\kappa(x)$ is a domain wall function. By assumption, $\mathcal{D}_\kappa$ satisfies the hypotheses of Theorem \ref{th:on_dirac}, and since $\lim_{x \rightarrow \infty} \kappa(x) = \kappa_\infty > 0$ and $\lim_{x \rightarrow - \infty} \kappa(x) = - \kappa_\infty$, the one domain wall Dirac operator has a unique zero mode given by \eqref{eq:zero_mode} (see Figure \ref{fig:one_domain_wall_zero_mode}). 
\begin{figure}
\includegraphics[scale=.6]{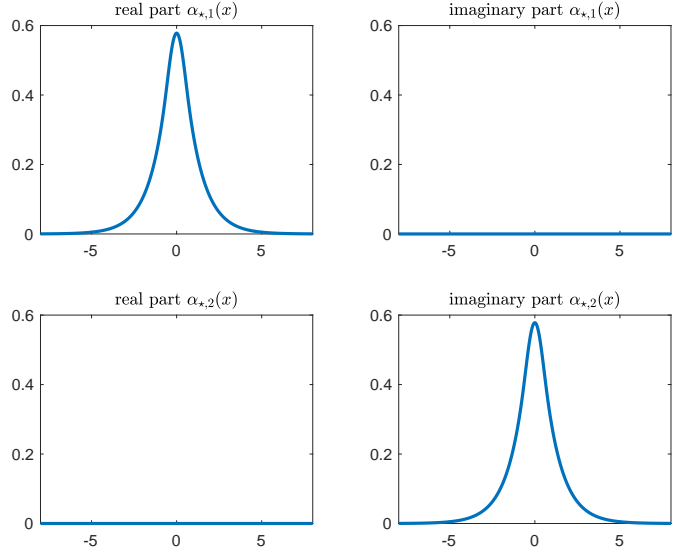}
\caption{The zero mode $\alpha_\star(x) = (\alpha_{\star,1}(x),\alpha_{\star,2}(x))^{\top}$ of the operator \eqref{eq:single_dom_wall_operator} with $\kappa$ as in Figure \ref{fig:kappa}, given analytically by \eqref{eq:zero_mode}. }
\label{fig:one_domain_wall_zero_mode}
\end{figure}
We will work with the following ``spectral gap'' assumption on the operator $\mathcal{D}_{\kappa}$. 
This assumption may be significantly weakened; see Remark \ref{rem:remark_on_weaker_assumptions}. 

\begin{assumption}[Spectral gap] \label{as:assumption_on_D}
Let $E = 0$ be the unique eigenvalue of $\mathcal{D}_\kappa$ in the spectral gap $(-\kappa_\infty,\kappa_\infty)$. That is, for all $f\in \mathcal{H}^1$ such that $\ip{\alpha_\star}{f}_\mathcal{H} = 0$, then: 
\begin{equation}
	\| \mathcal{D}_\kappa f \|_\mathcal{H} \geq \kappa_\infty \| f \|_\mathcal{H}.
\label{lb-Dk}\end{equation}
\end{assumption}
Assumption \ref{as:assumption_on_D} holds, for example, when $\kappa_\infty = 1$ and $\kappa(x) = \tanh(x)$ (which satisfies the relaxed domain wall function conditions \eqref{eq:weaker_1}-\eqref{eq:weaker_2}); we give a self-contained proof of this in Appendix \ref{sec:tanh_unique_evalue}.
\begin{remark} \label{rem:remark_on_weaker_assumptions} Our methods
  extend to the case where the operator $\mathcal{D}_{\kappa}$ has a
  finite number of point eigenvalues in the interval
  $(-\kappa_\infty,\kappa_\infty)$. Any such spectrum must be bounded
  a fixed distance $r > 0$ away from zero since the $0$-eigenvalue
  is simple. In this case, the bound \eqref{lb-Dk} is replaced
  by
  $\| \mathcal{D}_\kappa f \|_\mathcal{H} \geq r \| f \|_\mathcal{H}$
  and in the proof all estimates hold with $r$ in place of
  $\kappa_\infty$. For instance, our main result Theorem \ref{th:main_theorem} would then guarantee (for sufficiently large $\delta$) the existence of precisely two eigenvalues of $\mathcal{D}_{\kappa^\d}$ within any compact interval $[-K,K] \subset (-r,r)$. 
\end{remark}
An immediate consequence of Assumption \ref{as:assumption_on_D} is  the following: 
\begin{proposition} \label{prop:prop_on_D} Let Assumption
  \ref{as:assumption_on_D} hold. Choose any $E$ satisfying
  $|E| < \kappa_\infty$.  Then,
  \begin{enumerate}
  \item  if
  $f\in \mathcal{H}^1$ and
  $\ip{\alpha_\star}{f}_\mathcal{H} = 0$, then 
\begin{equation}
	\| ( \mathcal{D}_\kappa - E) f \|_\mathcal{H} \geq (\kappa_\infty - |E|) \| f \|_\mathcal{H}. 
\end{equation}
\item  Introduce the orthogonal projection $P^\perp: L^2(\mathbb{R}^2)\to {\rm span} \{\alpha_\star\}^\perp$ and let $\phi \in {\rm span}\{\alpha_\star\}^\perp$. 
Then, the equation
\begin{equation}
 P^\perp\ (\mathcal{D}_{\kappa} - E) \psi\ =\ \phi\ 
 \label{ihom0-E}\end{equation}
has a unique solution 
$\psi\in \mathcal{H}^1 \cap {\rm span}\{\alpha_\star\}^\perp$. We denote this solution by 
  $\psi\ =\ P^\perp\ (\mathcal{D}_{\kappa} - E)^{-1}\ P^\perp\ \phi$. 
\medskip

\item The operator $P^\perp\ (\mathcal{D}_{\kappa} - E)^{-1}\ P^\perp$ satisfies the bound: 
\begin{equation} \label{eq:resolvent_bound}
	\| P^\perp\ (\mathcal{D}_{\kappa} - E)^{-1}\ P^\perp \|_{\mathcal{H} \rightarrow \mathcal{H}} \leq \frac{1}{\kappa_\infty - |E|}. 
\end{equation}
\end{enumerate}
\end{proposition}
\subsection{Zero modes of ``shifted'' one domain wall operators}
Since $\mathcal{D}_{\kappa^\delta}$ involves well-separated spatial shifts of the domain wall function, $\kappa$, we introduce  ``shifted'' one domain wall Dirac operators:
\begin{equation} \label{eq:shifted_operators}
	\mathcal{D}^\+_{\kappa} := i \sigma_3 \de_{x} + \kappa(x - \d) \sigma_1, \quad \mathcal{D}^\-_{\kappa} := i \sigma_3 \de_{x} - \kappa(x + \d) \sigma_1.
\end{equation}
The superscripts $R$ and $L$ refer, respectively,  to ``right'' and ``left''. Note that restricted for functions supported on $\{x>0\}$,
  $\mathcal{D}_{\kappa^\delta}= \mathcal{D}^\+_{\kappa}$ and on $\{x<0\}$, $\mathcal{D}_{\kappa^\delta}= \mathcal{D}^\-_{\kappa}$.

The shifted operators $\mathcal{D}^\+_{\kappa}$ and $ \mathcal{D}^\-_{\kappa}$ have zero modes, which are expressible in terms of the zero mode of $\mathcal{D}_{\kappa}$:
\begin{proposition} \label{prop:shifted}
Define 
\begin{equation} \label{eq:def_alpha_pm}
	\alpha_{\star}^{\+}(x) := \alpha_\star(x - \d), \quad \alpha_{\star}^{\-}(x) := \overline{ \alpha_\star(x + \d) } \ .
\end{equation}
Then, 
\[\mathcal{D}^\+_{\kappa}\alpha_\star^\+=0\qquad {\rm and}\qquad  \mathcal{D}^\-_{\kappa}\alpha_\star^\- =0\ .\]
\end{proposition}
Note that the operators $\mathcal{D}^\+_{\kappa}$ and
$ \mathcal{D}^\-_{\kappa}$ and the states $\alpha_{\star}^\+(x)$ and
$\alpha_{\star}^\-(x)$ depend on the separation parameter $\delta$. We
 suppress this dependence to avoid cluttering notation.
\begin{proof}
\begin{align*}
	&\left( i \sigma_3 \de_x + \kappa(x) \sigma_1 \right) \alpha_\star(x) = 0 & \nonumber	\\
	\implies &\left( i \sigma_3 \de_{x} + \kappa(x - \d) \sigma_1 \right) \alpha_\star(x - \d) = 0. & \text{(changing variables, $\de_x = \de_{x - \d}$)}  \nonumber	\\
\end{align*}
\begin{align*}
	\nonumber &\left( i \sigma_3 \de_x  + \kappa(x) \sigma_1 \right) \alpha_\star(x) = 0 &	\\
	\nonumber \implies &\left(  i \de_{x}\sigma_3 + \kappa(x + \d) \sigma_1 \right) \alpha_\star(x + \d) = 0 &\text{(changing variables, $\de_x = \de_{x + \d}$) } 	\\
	\nonumber \implies &\left( -  i \de_{x}\sigma_3 - \kappa(x + \d) \sigma_1 \right) \alpha_\star(x + \d) = 0 &\text{(multiply by $-1$)} 	\\
	\nonumber \implies &\left(  i \de_{x}\sigma_3 - \kappa(x + \d) \sigma_1 \right) \overline{ \alpha_\star(x + \d) } = 0. &\text{(complex conjugate, $\kappa$ real) } 
\end{align*}
\end{proof}
We remark at this point that the ``shifted'' one domain wall zero modes $\alpha_\star^\+$ and $\alpha_\star^\-$ are approximate zero modes of the two domain wall operator $\mathcal{D}_{\kappa^\delta}$ in the following sense:
\begin{proposition} \label{prop:app_zm}
For $\d > 1$ the ``shifted'' one domain wall zero modes $\alpha_\star^\+$ and $\alpha_\star^\-$ satisfy the estimates:
\begin{equation}
	\| \mathcal{D}_{\kappa^\d} \alpha_\star^\+ \|_\mathcal{H} \leq C e^{- 2 \kappa_\infty \d}, \quad \| \mathcal{D}_{\kappa^\d} \alpha_\star^\- \|_\mathcal{H} \leq C e^{- 2 \kappa_\infty \d}
\end{equation}
for constants $C > 0$ depending only on $\kappa$. Here $\mathcal{D}_{\kappa^\d}$ denotes the two domain wall operator \eqref{eq:kappa_L}.
\end{proposition}
For the proof of Proposition \ref{prop:app_zm}, see Appendix \ref{sec:proofs_0}.

We are now in a position to state our theorem for the two domain wall Dirac operator:
\begin{theorem} \label{th:main_theorem}
Let the Spectral Gap Assumption \ref{as:assumption_on_D} hold, and pick any compact interval $[-K,K] \subset (-\kappa_\infty,\kappa_\infty)$. Then there is a constant $\d_0(K)>0$ such that for all $\delta > \delta_0(K)$ the operator $\mathcal{D}_{\kappa^\d}$ defined in \eqref{eq:def_dirac} has precisely two simple eigenvalues $E^\delta_\pm$ in the interval $[-K,K]$. These eigenvalues have expansions: 
\begin{equation} \label{eq:E_asymptotics}
	E^\delta_\pm = \pm 2 \gamma^2 e^{- 2 \inty{0}{\d}{\kappa(y)}{y}} + O(e^{- 4 \kappa_\infty \d}). 
\end{equation}
Their associated (normalized) eigenfunctions, which we denote $\alpha^\d_\pm(x)$, may be written as approximate linear combinations of $\alpha^{\+,\d}_\star(x), \alpha^{\-,\d}_\star(x)$:
\begin{equation} \label{eq:two_domain_wall_modes}
\begin{split}
	&\alpha^\d_+(x) = \frac{\gamma}{\sqrt{2}} \left( \alpha^{\+,\d}_\star(x) + i \alpha^{\-,\d}_\star(x) \right) + O_{\mathcal{H}}(e^{- 2 \kappa_\infty \d})	\\
	&\alpha^\d_-(x) = \frac{\gamma}{\sqrt{2}} \left( \alpha^{\+,\d}_\star(x) - i \alpha^{\-,\d}_\star(x) \right) + O_{\mathcal{H}}(e^{- 2 \kappa_\infty \d})	
\end{split}		
\end{equation}
where the functions $\alpha^{\+,\d}_\star(x), \alpha^{\-,\d}_\star(x)$ are the shifted zero-mode functions defined by \eqref{eq:def_alpha_pm} and the real constant $\gamma$ is as in \eqref{eq:zero_mode} . 
\end{theorem}
Note that in the statement of the theorem we make explicit the dependence of the functions $\alpha^\+_\star$, $\alpha_\star^\-$, and $\alpha_\pm$ on $\delta$. 
A numerical computation of the modes \eqref{eq:two_domain_wall_modes} is displayed in Figure \ref{fig:two_domain_walls_modes}. 
\begin{figure}
\includegraphics[scale=.6]{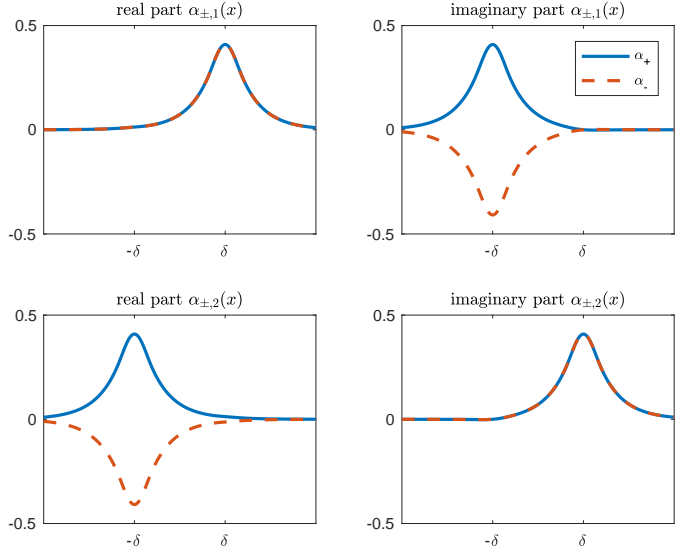}
\caption{Numerical computation of the two near-zero modes \eqref{eq:two_domain_wall_modes} of the operator $\mathcal{D}_{\kappa^\d}$ \eqref{eq:def_dirac} for $\kappa^\delta$ as in Figure \ref{fig:kappa_L}. Note $\Re\alpha_{+,1}=\Re\alpha_{-,1}$,
 and $\Im\alpha_{+,2}=\Im\alpha_{-,2}$ }
\label{fig:two_domain_walls_modes}
\end{figure}


We outline the strategy of the proof of Theorem \ref{th:main_theorem} in Section \ref{sec:strategy_of_proof}, postponing the proofs of key propositions to Appendix \ref{sec:proofs}.  
\begin{remark}
If we relax the assumption that $\kappa$ is odd (part (2) of Assumption \ref{s-dw}), the leading order behavior of the near-zero eigenvalues \eqref{eq:E_asymptotics} must be modified to the form: 
\begin{equation} \label{eq:E_without_oddness}
	E^\delta_\pm = \pm 2 \gamma^2 e^{- \inty{0}{\d}{\kappa}{y}} e^{\inty{-\d}{0}{ \kappa}{y}} + O(e^{- 4 \kappa_\infty \d})	
\end{equation}
where the real constant $\gamma$ is as in \eqref{eq:zero_mode}.
\end{remark}

For the case of $n$ domain walls, the generalization of Theorem \ref{th:main_theorem} which follows from a similar analysis is as follows:
\begin{theorem} \label{th:n_theorem}
Let the Spectral Gap Assumption \ref{as:assumption_on_D} hold, and pick any compact interval $[-K,K] \subset (- \kappa_\infty,\kappa_\infty)$. Let $\mathcal{D}_{\kappa^\d}$ denote an $n$ domain wall Dirac operator obtained by ``glue-ing'' $n$ domain walls together, each a distance $2 \d$ apart as in the definition of the ``two domain wall function'' \eqref{eq:kappa_L}. Then, there is a constant $\delta_0(K)$ such that for all $\delta > \delta_0(K)$, the operator $\mathcal{D}_{\kappa^\d}$ has precisely $n$ simple eigenvalues $E^{\d}_j$, $j \in \{1,...,n\}$ in the interval $[-K,K]$. These eigenvalues have expansions $E^\d_j = E^{\d}_{j,0} + O(e^{- 4 \kappa_\infty \d})$, where $E^{\d}_{j,0}$ denotes the $j$th eigenvalue of the $n \times n$ tri-diagonal matrix:
\begin{equation}
   M_0(\delta) := \begin{pmatrix} 0 & - 2 i \gamma^2 e^{- 2 \inty{0}{\d}{ \kappa(y) }{y} } & 0 & \hdots \\
   2 i \gamma^2 e^{- 2 \inty{0}{\d}{ \kappa(y) }{y} } & 0 & - 2 i \gamma^2 e^{- 2 \inty{0}{\d}{ \kappa(y) }{y}} & \hdots    \\
   0 & 2 i \gamma^2 e^{- 2 \inty{0}{\d}{ \kappa(y) }{y} } & 0 & \hdots  \\
   \vdots  & \vdots & \vdots & \ddots 
   \end{pmatrix}.
\end{equation}
The associated (normalized) eigenfunctions $\alpha^{\d}_j$ of these eigenvalues have expansions $\alpha^{\d}_j = \alpha^{\d}_{j,0} + O_{\mathcal{H}}(e^{- 2 \kappa_\infty \d})$ whose leading order terms $\alpha^{\d}_{j,0}$ are linear combinations of shifted copies of the zero-mode function $\alpha_\star$ \eqref{eq:zero_mode}. The precise linear combinations are determined by the associated eigenvectors of $M_0(\delta)$ corresponding to the eigenvalues $E_{j,0}^\d$. 
\end{theorem}
We give details of the proof of Theorem \ref{th:n_theorem} for the case $n = 3$ in Sections \ref{sec:3_dom_wall_case} and \ref{sec:proof_3dw_lemma}, but omit the proof for general $n$ since the argument is essentially identical.

Whenever $n$ is odd, because $\lim_{x \rightarrow \infty}\kappa^\d = \kappa_\infty$ and $\lim_{x \rightarrow - \infty} \kappa^\d = - \kappa^\infty$, $\mathcal{D}_{\kappa^\d}$ has a unique (up to a complex constant of norm $1$) \emph{exact} normalized zero mode given by
\begin{equation} \label{eq:n_dw_zero_mode}
	\alpha_\star^\d(x) := \gamma^\d \begin{pmatrix} 1 \\ i \end{pmatrix} e^{ - \inty{0}{x}{ \kappa^\d(y) }{y} },\qquad 
	\gamma^\d := \frac{1}{\sqrt{2} \| e^{- \inty{0}{x}{ \kappa^\d(y) }{y} } \|_{L^2} }. 
\end{equation}
It follows that one of the eigenvalues $E^\d_{j}$ within the interval $[-K,K]$ whose existence is guaranteed by Theorem \ref{th:n_theorem} must correspond to this mode. This occurs, for example, in the case of three domain walls. We discuss this in Section \ref{sec:3_dom_wall_case}. 

\begin{remark} 
\label{rem:remark_on_nonequal_spacing} The results we have discussed so far treat only the case where all of the domain walls are equally spaced. Our analysis does not rely on this, only on the minimal distance between neighbouring domain walls being large. However, in the case that they are not equally
  spaced, we expect the result to be rather complicated to state for
  the following reason. When all domain walls are equally spaced from
  each other by a distance $2 \d$, the small parameter
  $e^{- 2 \kappa_\infty \d}$ emerges naturally in expansions of the
  eigenvalues and eigenfunctions. If domain walls spaced from each
  other by, for example, distances $2 \d$ and $2 \d'$ where
  $\d \neq \d'$ are allowed, expansions of the eigenvalues and
  eigenfunctions will instead depend on powers of both
  $e^{- 2 \kappa_\infty \d}$ and $e^{- 2 \kappa_\infty \d'}$. Hence,
  we expect the general result in this case to be rather complicated
  to state, although no new technical difficulty arises. 
\end{remark}


\section{Proof of Theorem \ref{th:main_theorem} on near-zero energy bound states of the two domain wall operator (strategy)} \label{sec:strategy_of_proof}
We now describe the strategy of the proof of Theorem \ref{th:main_theorem}. We start by seeking a solution of the eigenvalue problem (\ref{eq:eigenvalue_problem}) as a linear combination of the shifted single domain wall states $\alpha^\+_\star, \alpha^\-_\star$ plus a corrector function $\eta$ orthogonal to $\alpha^\+_\star$ and $ \alpha^\-_\star$ :
\begin{equation} \label{eq:ansatz}
	\alpha(x) = b^\+ \alpha_\star^\+(x) + b^\- \alpha_\star^\-(x) + \eta(x),\quad \ip{\alpha^I_\star}{\eta}=0,\ \ I=\+,\-\ .
\end{equation}
The constants $b^\+,  b^\-\in\mathbb{C}$, together with $\eta$ and $E$ are to be determined. 
 As earlier, we suppress the $\d-$ dependence of terms in \eqref{eq:ansatz}.

Remarkably, the  decomposition in \eqref{eq:ansatz} is an orthogonal decomposition:
\begin{lemma}\label{orthog-sym} 
\begin{equation}
	\ip{\alpha_\star^I}{\alpha_\star^J}_\mathcal{H} =\ \delta_{IJ},\qquad I, J\in\{\+,\-\}.
	\label{pm-orth}
\end{equation}
\end{lemma}
\begin{proof}
The assertion for $I=J$ is immediate from the definitions of $\alpha_\star^\-, \alpha_\star^\+$; see \eqref{eq:def_alpha_pm}.
Now consider the case $I \ne J$. To see that 
$	\ip{\alpha_\star^\-}{\alpha_\star^\+}_\mathcal{H} = 0 $, note 
\begin{align}
	\nonumber &\ip{\alpha_\star^\-}{\alpha_\star^\+}_\mathcal{H} = 
	\inty{\field{R}}{}{ \ip{\alpha_\star^\-(x)}{\alpha_\star^\+(x)} }{x}&	\\
	\nonumber &= \inty{\field{R}}{}{ \ip{ \overline{ \alpha_\star(x + \d) } }{ \alpha_\star(x - \d)} }{x}	&\text{(by: (\ref{eq:def_alpha_pm}))}	\\
	&= \inty{\field{R}}{}{ \gamma^2 \ip{ \begin{pmatrix} 1 \\ - i \end{pmatrix} }{\begin{pmatrix} 1 \\ i \end{pmatrix}} e^{- \inty{0}{x + \d}{ \kappa(y) }{y}} e^{- \inty{0}{x - \d}{ \kappa(y) }{y}} }{x} = 0. &\text{(by: (\ref{eq:zero_mode}))}  \label{eq:cancellation} 
\end{align}
This completes the proof.
\end{proof}
\begin{remark}\label{not-orth}
Even without the exact cancellation \eqref{eq:cancellation}, $\ip{\alpha_\star^\-}{\alpha_\star^\+}_\mathcal{H}$ is already $O(e^{- 2 \kappa_\infty \d})$ for $\delta$ sufficiently large. Below, when we treat the 3-domain wall case, the analogous decomposition is no longer orthogonal and we must make use of exponential decay of the analogous inner products instead (see \eqref{eq:bound_1_prime}). We shall explain the required modifications in the proof. 
\end{remark}

Let $\Prl$ denote the orthogonal projection in $\mathcal{H}$ onto the subspace $\Hrl \equiv \{ \alpha^\+_{\star}, \alpha^\-_{\star}\}^\perp$. 
We recall that $\alpha^\+_{\star}$ and $ \alpha^\-_{\star}$ and hence $\mathcal{H}$, $\Hrl$ and $\Prl$ depend on $\delta$, but we suppress this dependence.

Since the decomposition $\mathcal{H}={\rm span}\{ \alpha^\+_\star \}\oplus{\rm span}\{\alpha^\-_\star \}\oplus\Hrl$ is a orthogonal one, we can obtain an equivalent formulation of the eigenvalue problem by substituting \eqref{eq:ansatz} into \eqref{eq:eigenvalue_problem} and orthogonally projecting onto each subspace. This gives a coupled system of three equations for $b^\+, b^\-$ and $\eta$ which depends on the eigenvalue parameter $E$ and domain wall separation parameter $\d$:
\begin{align}
	& \sum_{j \in \+,\-} b^j \ip{ \alpha_\star^i }{ (\mathcal{D}_{\kappa^\d} - E) \alpha_\star^j }_\mathcal{H} + \ip{ \alpha_\star^i }{ (\mathcal{D}_{\kappa^\d} - E) \eta }_\mathcal{H} = 0,\qquad \text{$i= \+,\-$, } \label{eq:parallel_equation}	\\
	&\sum_{j \in \+,\-} b^j \Prl(\mathcal{D}_{\kappa^\d} - E) \alpha_\star^j + \Prl (\mathcal{D}_{\kappa^\d} - E) \eta = 0. \label{eq:perp_equation}
\end{align}

The next step (Lyapunov-Schmidt reduction) is to solve \eqref{eq:perp_equation} for $\eta$ as a function of $b^\+$ and $b^\-$. The mapping $(b^\+,b^\-)\mapsto \eta[b^\+,b^\-;E,\d]$ is linear in $b^\+$ and $b^\-$ and depends nonlinearly on $E$ and $\d$. Substitution of this mapping into equations \eqref{eq:parallel_equation} gives a system of two linear homogeneous equations: $M(E,\d)b=0$ for the unknowns  $b=(b^\+,b^\-)^{\top}$, depending nonlinearly on the energy $E$ and $\d$. We then proceed
to solve this equation for $E=E(\d)$, in a neighborhood of $E=0$, for all $\d$ sufficiently large. 

The reduction step is facilitated by the following:
\begin{proposition}[Basic Energy Estimate] \label{prop:prop_on_DL}
Fix $K > 0$ such that $K < \kappa_\infty$. Then there exists a $\d_0(K) \geq 2$, sufficiently large,  such that for all $\d > \d_0(K)$ the following holds: If $f\in \mathcal{H}^1$ satisfies $\ip{\alpha_\star^\+}{f}_\mathcal{H} = \ip{\alpha_\star^\-}{f}_\mathcal{H} = 0$, then
\begin{equation} \label{eq:statement_lemma_on_DL}
	\| \mathcal{D}_{\kappa^\d} f \|_\mathcal{H} \geq \frac{\kappa_\infty + K}{2} \| f \|_\mathcal{H}. 
\end{equation}
Moreover, for $|E| \leq K$:
\begin{equation} \label{eq:second_part_lem_on_DL}
	\| (\mathcal{D}_{\kappa^\d} - E) f \|_\mathcal{H} \geq \frac{\kappa_\infty - K}{2}  \| f \|_\mathcal{H}. 
\end{equation}
\end{proposition} 

%
Proposition \ref{prop:prop_on_DL} is proved by expressing the energy associated with $\mathcal{D}_{\kappa^\d}$, via a partition of unity, as a superposition of localized energies localized near and away from domain wall ``cores''. The near-core energies are controlled using  Proposition \ref{prop:prop_on_D} and the energy away from the domain wall cores is essentially the energy associated with 
 the constant coefficient operators $\mathcal{D}_{\pm\kappa_\infty}$.
Variants of this technique are used, for example, in \cite{CyconFroeseKirschSimon} (Chapter 3) and \cite{FLW-CPAM:17}.
The detailed proof of Proposition \ref{prop:prop_on_DL} is given in  Appendix \ref{sec:proofs_1}. 

\medskip

\begin{corollary}\label{F0} 
\begin{enumerate}
\item Let $\phi \in \Hrl$ and $|E| \leq K$ where $0 < K < \kappa_\infty$. Then the equation
\begin{equation}
 \Prl\ (\mathcal{D}_{\kappa^\d} - E) \psi\ =\ \phi\ 
 \label{ihom-E}\end{equation}
has a unique solution 
$\psi\in \mathcal{H}^1 \cap \Hrl$. 
\medskip

\item The operator $\Prl\ (\mathcal{D}_{\kappa^\d} - E)^{-1}\ \Prl$ satisfies the bound:
\begin{equation} 
	\| \Prl\ (\mathcal{D}_{\kappa^\d} - E)^{-1}\ \Prl \|_{\mathcal{H} \rightarrow \mathcal{H}} \leq \frac{2}{\kappa_\infty - K}. 
\end{equation}
\end{enumerate}
\end{corollary}
\begin{proof} Part (2) is a simple  consequence of part (1) and Proposition \ref{prop:prop_on_DL}. 
%
%
\ To prove part (1), we first show the solvability of \eqref{ihom-E} for $E=0$. Then solvability of \eqref{ihom-E} for $|E| \leq K$ follows by a perturbation argument. Suppose that there exists $\psi_0\in \Hrl$ which is not the image 
under $\Prl \mathcal{D}_{\kappa^\d}$ of a function in $\mathcal{H}^1 \cap \Hrl$. Then,  for all $\psi \in \mathcal{H}$ we have 
$\ip{\psi_0}{\Prl \mathcal{D}_{\kappa^\d} \Prl\psi}_{\mathcal{H}}=\ip{\psi_0}{\mathcal{D}_{\kappa^\d} \Prl\psi}_{\mathcal{H}}=0$. 

Therefore, $\psi_0$ is a weak solution of $\Prl \mathcal{D}_{\kappa^\d}\psi_0=0$ and by elliptic regularity 
$\psi_0 \in \mathcal{H}^2 \cap \Hrl$. Therefore, $\ip{\Prl \mathcal{D}_{\kappa^\d} \psi_0}{\psi}_{\mathcal{H}}=0$ for all
 $\psi\in \mathcal{H}^1 \cap \Hrl$. Taking $\psi = \Prl \mathcal{D}_{\kappa^\d} \psi_0$, we have that $\| \Prl \mathcal{D}_{\kappa^\d} \psi_0 \|^2_{\mathcal{H}} = 0$. To see that this implies that $\psi_0 = 0$, consider that by definition and self-adjointness of $ \mathcal{D}_{\kappa^\d}$:
\begin{equation} \label{eq:simple}
	\Prl \mathcal{D}_{\kappa^\d} f = \mathcal{D}_{\kappa^\d} f - \sum_{j = R,L} \ip{\mathcal{D}_{\kappa^\d}\alpha_\star^j}{ f}_\mathcal{H} \alpha_\star^j, \quad f \in \mathcal{H}^1.
\end{equation}
The ${\mathcal{H}}-$ norm of the second term on the right-hand side of \eqref{eq:simple} is $O(e^{- 2 \kappa_\infty \d}\ \|f\|_{\mathcal{H}})$ since $\| \mathcal{D}_{\kappa^\d} \alpha^j_{\star} \|_\mathcal{H} = O(e^{- 2 \kappa_\infty \d})$ (Proposition \ref{prop:app_zm}).
%
 Using this bound and then Proposition \ref{prop:prop_on_DL} we obtain, for $\delta$ sufficiently large:
 \begin{align*}
 \| \Prl \mathcal{D}_{\kappa^\d} \psi_0\ \|_\mathcal{H} &\ge \| \mathcal{D}_{\kappa^\d} \psi_0 \|_\mathcal{H} - C\ e^{- 2 \kappa_\infty \d}\ \|\psi_0\|_{\mathcal{H}}\\
 &\ge \left(\ \frac{\kappa_\infty + K}{2}-C\ e^{- 2 \kappa_\infty \d}\ \right)\|\psi_0\|_\mathcal{H}\ \ge\ \frac{\kappa_\infty - K}{2} \ \|\psi_0\|_\mathcal{H}.  \end{align*} 
%
%
Thus,  $\| \Prl \mathcal{D}_{\kappa^\d} \psi_0 \|_{\mathcal{H}} = 0$ implies that $\psi_0 = 0$ in $\mathcal{H}$.
\end{proof}
Assuming Proposition \ref{prop:prop_on_DL}, for $|E| \leq K$ we solve (\ref{eq:perp_equation}) in terms of $b^\+, b^\-$:
\begin{equation} \label{eq:eta}
\begin{split}
	\eta &= - \sum_{j \in \+,\-} b^j \Prl (\mathcal{D}_{\kappa^\d} - E)^{-1} \Prl (\mathcal{D}_{\kappa^\d} - E) \alpha_\star^j	\\
	&= - \sum_{j \in \+,\-} b^j \Prl (\mathcal{D}_{\kappa^\d} - E)^{-1} \Prl \mathcal{D}_{\kappa^\d} \alpha_\star^j.	
\end{split}
\end{equation}
Substituting \eqref{eq:eta} back into (\ref{eq:parallel_equation}), yields a \emph{closed} system for $b^\+, b^\-$, which depends on $E$ and $\delta$:
\begin{multline} \label{eq:system_for_bs}
  \text{$i \in \+,\-$, } \sum_{j \in \+,\- } \ip{ \alpha_\star^i }{ (\mathcal{D}_{\kappa^\d} - E) \alpha_\star^j }_\mathcal{H} b^j  \\
	- \sum_{j \in \+,\-} \ip{ \alpha_\star^i }{ (\mathcal{D}_{\kappa^\d} - E) \Prl (\mathcal{D}_{\kappa^\d} - E)^{-1} \Prl \mathcal{D}_{\kappa^\d} \alpha_\star^j }_\mathcal{H} b^j = 0.
\end{multline}
Equation (\ref{eq:system_for_bs}) can be written as the following homogeneous system:
\begin{equation}\label{eq:reduced_equation_for_bs}
	\sum_{j \in \+,\-} M^{i j} b^j = 0,\quad \text{$i \in \+,\-$, }  	
	\end{equation}
	where by self-adjointness of $\Prl (\mathcal{D}_{\kappa^\d} - E) \Prl $,  we have
\begin{equation}
M^{ij}(\d,E) := \ip{ \alpha_\star^i }{ \mathcal{D}_{\kappa^\d} \alpha_\star^j }_\mathcal{H} - E \ip{ \alpha_\star^i }{ \alpha_\star^j }_\mathcal{H} - \ip{ \Prl \mathcal{D}_{\kappa^\d} \alpha_\star^i }{ (\mathcal{D}_{\kappa^\d} - E)^{-1} \Prl \mathcal{D}_{\kappa^\d} \alpha_\star^j }_\mathcal{H}.
\end{equation}
In particular, we have:
\begin{corollary} \label{cor:corollary_on_DL}
Fix $E$ such that $|E| \leq K$. There exists $\delta_0(K) \geq 2$ sufficiently large such that for all $\delta>\delta_0(K)$,
\begin{align}
&\textrm{ $E$ is an eigenvalue of $\mathcal{D}_{\kappa^\d}$ if and only if $\text{\emph{det }} M^{i j}(\d,E) = 0$.}
\end{align} 
\end{corollary}
Theorem \ref{th:main_theorem} will follow from a detailed analysis of each component of the matrix $M^{ij}(\d,E)$, assuming that $|E| \leq K$ so that \eqref{eq:second_part_lem_on_DL} and Corollary \ref{cor:corollary_on_DL} hold. Note that the resolvent operator $\Prl (\mathcal{D}_{\kappa^\d} - E)^{-1} \Prl$ is actually analytic in $E$ in the complex ball of radius $K$ 
centered at the origin. The following proposition summarizes the result of our analysis of the matrix $M(\d,E)$:
\begin{proposition} \label{lem:bounds_on_terms_in_M}
Assume that $|E| \leq K$ so that Proposition \ref{prop:prop_on_DL} (in particular \eqref{eq:second_part_lem_on_DL}) holds. Then each of the entries of $M(\d,E)$ varies analytically with $E$, and the matrix $M(\d,E)$ may be written as:
\begin{equation} \label{eq:M_ij_with_bounds}
	M(\d,E) = \begin{pmatrix} - E & - 2 i \gamma^2 e^{- 2 \inty{0}{\d}{\kappa(y)}{y}} \\ 2 i \gamma^2 e^{- 2 \inty{0}{\d}{\kappa(y)}{y}} & - E \end{pmatrix} + M_1(\d,E)
\end{equation}
where each of the entries of the matrix $M_1(\d,E)$ satisfies: 
\begin{equation} \label{eq:resolvent_bounded}
	| M^{ij}_1(\d,E) | \leq C e^{- 4 \kappa_\infty \d}
\end{equation}
for some constant $C > 0$ independent of $\d, E$. 
\end{proposition}
 See Appendix \ref{sec:proofs_2} for the proof of Proposition~\ref{lem:bounds_on_terms_in_M}. We are now in a position to prove Theorem \ref{th:main_theorem}, which will follow from Corollary \ref{cor:corollary_on_DL} and a careful analysis of the roots of the determinant of the matrix $M(\d,E)$ appearing in \eqref{eq:M_ij_with_bounds}. Using \eqref{eq:resolvent_bounded}, we may write the determinant as: 
\begin{equation} \label{eq:det_M}
	\det M(\d,E) = ( E - 2 \gamma^2 e^{- 2 \inty{0}{ \d}{\kappa(y)}{y}} ) ( E + 2 \gamma^2 e^{- 2 \inty{0}{ \d}{\kappa(y)}{y}} ) + g(\d,E),
\end{equation}
where the function $g(\d,E)$ is analytic in $E$ and satisfies the bound: 
\begin{equation} \label{eq:g_est}
	| g(\d,E) | \leq C_1 \left( E e^{- 4 \kappa_\infty \d} + e^{- 6 \kappa_\infty \d} \right),
\end{equation}
where $C_1 > 0$ is a constant independent of $\d$ and $E$. Our strategy is as follows. We will first bound \eqref{eq:g_est} along contours in the complex plane centered at $\pm 2 \gamma^2 e^{- 2 \inty{0}{\d}{\kappa(y)}{y}}$ and then apply Rouch\'{e}'s theorem to conclude that $\det M(\d,E)$ has, for sufficiently large $\d$, precisely one root within each contour. The asymptotic expressions for the eigenvalues and associated eigenfunctions of $\mathcal{D}_{\kappa^\d}$ \eqref{eq:E_asymptotics}-\eqref{eq:two_domain_wall_modes} will then follow from an analysis of the asymptotic behavior of these roots. Finally, we will show uniqueness of the eigenvalues of $\mathcal{D}_{\kappa^\d}$ within the ball $|E| \leq K$ via a second application of Rouch\'{e}'s theorem. 

Note first that it is clear from \eqref{eq:properties_of_kappa} that there exist constants $C_2, C_3 > 0$, independent of $\d, E$ but dependent on $\kappa$, such that:
\begin{equation} \label{eq:basic_est}
	C_2 e^{- 2 \kappa_\infty \d} < 2 \gamma^2 e^{- 2 \inty{0}{ \d}{\kappa(y)}{y}} < C_3 e^{- 2 \kappa_\infty \d}
\end{equation}
for any $\d \geq 2$. We consider the contours:
\begin{equation} \label{eq:contours}
	\gamma_\pm := \left\{ \pm 2 \gamma^2 e^{- 2 \inty{0}{ \d}{\kappa(y)}{y}} + e^{i \theta} C_2 e^{- 2 \kappa_\infty \d} : \theta \in [0,2\pi] \right\}.
\end{equation}
It is clear from \eqref{eq:g_est} that $\left| g(\d,E) \right|$ may be uniformly bounded above for $E \in \gamma_\pm$ by a constant times $e^{- 6 \kappa_\infty \d}$. We now claim that the quadratic part of \eqref{eq:det_M} may be bounded below uniformly on the contours $\gamma_\pm$ \eqref{eq:contours} by a constant times $e^{- 4 \kappa_\infty \d}$. Without loss of generality since the contour $\gamma_-$ is similar, we show this for the contour $\gamma_+$ only. Evaluating the polynomial part of \eqref{eq:det_M} on the contour $\gamma_+$ gives: 
\begin{equation} \label{eq:polynomial_part}
\begin{split}
	&\left. ( E - 2 \gamma^2 e^{- 2 \inty{0}{\d}{\kappa(y)}{y}} )( E + 2 \gamma^2 e^{- 2 \inty{0}{\d}{\kappa(y)}{y}} ) \right|_{E \in \gamma_+} 	\\
	&= \left( e^{i \theta} C_2 e^{- 2 \kappa_\infty \d} \right)\left( 2 \left( 2 \gamma^2 e^{- 2 \inty{0}{\d}{\kappa(y)}{y}} \right) + e^{i \theta} C_2 e^{- 2 \kappa_\infty \d} \right) \quad \theta \in [0,2\pi].
\end{split}
\end{equation}
Applying the triangle inequality and then using \eqref{eq:basic_est} gives a lower bound on \eqref{eq:polynomial_part} which is uniform in $\theta$: 
\begin{equation}
\begin{split}
	&\geq C_2 e^{- 2 \kappa_\infty \d} \left( 2 \left( 2 \gamma^2 e^{- 2 \inty{0}{\d}{\kappa(y)}{y}} \right) - C_2 e^{- 2 \kappa_\infty \d} \right)	\\
	&> C_2 e^{- 2 \kappa_\infty \d} \left( 2 \gamma^2 e^{- 2 \inty{0}{\d}{\kappa(y)}{y}} \right) > C_2^2 e^{- 4 \kappa_\infty \d}.
\end{split}
\end{equation}
By Rouch\'{e}'s theorem, it now follows that for sufficiently large $\d > 2$, $\det M(\d,E)$ has the same number of roots within each of the contours $\gamma_\pm$ as its quadratic part, which has precisely one within each contour. 

Having established that $\det M(\d,E)$ has precisely one root within each contour, we now show that these roots satisfy the asymptotics \eqref{eq:E_asymptotics}. 
Again without loss of generality, we consider the root within the contour $\gamma_+$, denoting it by $E_+$. By definition \eqref{eq:det_M}, $E_+$ satisfies:
\begin{equation} \label{eq:def_E_+}
	( E_+ - 2 \gamma^2 e^{- 2 \inty{0}{ \d}{\kappa(y)}{y}} ) ( E_+ + 2 \gamma^2 e^{- 2 \inty{0}{ \d}{\kappa(y)}{y}} ) + g(\d,E_+) = 0.
\end{equation}
Since $E_+$ must lie in the interior of the contour $\gamma_+$ we have the bounds:
\begin{equation} \label{eq:cont_bds}
	2 \gamma^2 e^{- 2 \inty{0}{\d}{\kappa(y)}{y}} - C_2 e^{- 2 \kappa_\infty \d} \leq E_+ \leq 2 \gamma^2 e^{- 2 \inty{0}{\d}{\kappa(y)}{y}} + C_2 e^{- 2 \kappa_\infty \d}.
\end{equation}
Combining \eqref{eq:cont_bds} with \eqref{eq:basic_est} we then have that:
\begin{equation} \label{eq:E_+_est}
	0 \leq E_+ \leq (C_2 + C_3) e^{- 2 \kappa_\infty \d}.
\end{equation}
Since $E_+ \geq 0$ \eqref{eq:E_+_est}, we may divide by $E_+ + 2 \gamma^2 e^{- 2 \inty{0}{ \d}{\kappa(y)}{y}}$ in \eqref{eq:def_E_+} to obtain: 
\begin{equation} \label{eq:EE_++}
	E_+ - 2 \gamma^2 e^{- 2 \inty{0}{ \d}{\kappa(y)}{y}}  + \frac{ g(\d,E_+) }{ E_+ + 2 \gamma^2 e^{- 2 \inty{0}{ \d}{\kappa(y)}{y}} } = 0.
\end{equation}
Combining \eqref{eq:g_est} with \eqref{eq:E_+_est} and combining \eqref{eq:E_+_est} with \eqref{eq:basic_est} respectively yields the bounds:
\begin{equation} \label{eq:more}
\begin{split}
	&| g(\d,E_+) | \leq C_1^2 ( E_+ e^{- 4 \kappa_\infty \d} + e^{- 6 \kappa_\infty \d} ) = O(e^{- 6 \kappa_\infty \d})	\\
	&E_+ + 2 \gamma^2 e^{- 2 \inty{0}{ \d}{\kappa(y)}{y}} \geq C_2 e^{- 2 \kappa_\infty \d}.
\end{split}
\end{equation}
Using \eqref{eq:more} we can bound the third term appearing in \eqref{eq:EE_++} and derive the asymptotic expansions \eqref{eq:E_asymptotics}:
\begin{equation} \label{eq:E_as}
	E_\pm = \pm 2 \gamma^2 e^{- 2 \inty{0}{ \d}{ \kappa(y) }{ y } } + O(e^{- 4 \kappa_\infty \d}).
\end{equation}
Substituting these expressions into the matrix eigenvalue problem \eqref{eq:reduced_equation_for_bs} yields asymptotics for the associated eigenvectors:
\begin{equation} 
\begin{split}
	&\begin{pmatrix} b^\+ \\ b^\- \end{pmatrix}_+ = \frac{1}{\sqrt{2}} \begin{pmatrix} 1 \\ i \end{pmatrix} + O(e^{- 2 \kappa_\infty \d})	\\
	&\begin{pmatrix} b^\+ \\ b^\- \end{pmatrix}_- = \frac{1}{\sqrt{2}} \begin{pmatrix} 1 \\ - i \end{pmatrix} + O(e^{- 2 \kappa_\infty \d}),
\end{split}
\end{equation}
from which the asymptotics of the eigenfunctions of $\mathcal{D}_{\kappa^\d}$ \eqref{eq:two_domain_wall_modes} follow. 

To establish uniqueness of the eigenvalues $E_\pm(\d)$ \eqref{eq:E_as} within the interval $|E| \leq K$, we claim that the quadratic part of $\det M(\d,E)$ dominates the remainder $g(\d,E)$ uniformly over the contour $|E| = K$ and hence, again by Rouch\'{e}'s theorem, the determinant has precisely two roots within the contour. 
To see this, observe that for any $E$ the quadratic part of $\det M(\d,E)$ may be bounded below as follows: 
\begin{equation} \label{eq:uniqueness_1}
\begin{split}
	&( E - 2 \gamma^2 e^{- 2 \inty{0}{\d}{ \kappa(y) }{y}} )( E + 2 \gamma^2	e^{- 2 \inty{0}{\d}{ \kappa(y) }{y}} )	\\
	&= E^2 - \left( 2 \gamma^2 e^{- 2 \inty{0}{\d}{ \kappa(y) }{y}} \right)^2 \geq |E|^2 - C_2^2 e^{- 4 \kappa_\infty \d} 
\end{split}
\end{equation}
where the last inequality follows immediately from \eqref{eq:basic_est}. Using the bound \eqref{eq:g_est} on the remainder $g(\d,E)$ now implies that the quadratic part of the determinant dominates the remainder uniformly over the contour $|E| = {K}$ as long as:
\begin{equation} \label{eq:uniqueness_2}
	C_1 \left( {K} e^{- 4 \kappa_\infty \d} + e^{- 6 \kappa_\infty \d} \right) \leq {K}^2 - C_2^2 e^{- 4 \kappa_\infty \d} 
\end{equation}
which holds for any ${K} > 0$ for $\d$ sufficiently large.

\section{The case of three domain walls} \label{sec:3_dom_wall_case}
Theorem \ref{th:main_theorem}  may be generalized to the case where $\kappa^\d$ has the form of a ``3- domain wall'' mass function (see Figure \ref{fig:3_dw_kappa}):
\begin{equation} \label{eq:three_dw_kappa}
	\kappa^\d(x) = \begin{cases} \kappa(x + 2\d) & \text{for } - \infty \leq x \leq - \d \\ - \kappa(x) & \text{for } - \d \leq x \leq \d \\ \kappa(x - 2\d) & \text{for } \d \leq x \leq \infty \end{cases}.
\end{equation}
In this case, because $\lim_{x \rightarrow \infty} \kappa^\d = \kappa_\infty$ and $\lim_{x \rightarrow - \infty} \kappa^\d = - \kappa_\infty$, the operator $\mathcal{D}_{\kappa^\d}$ has a unique (up to a complex constant of norm 1) \emph{exact} normalized zero mode given by \eqref{eq:n_dw_zero_mode}. 
%
This mode is plotted in Figure \ref{fig:3_dw_zero_mode}. 
\begin{figure}
\includegraphics[scale=.35]{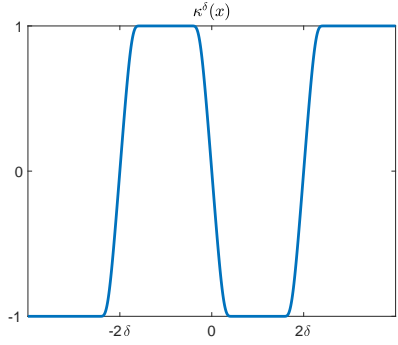}
\caption{Plot of $\kappa^\d(x)$ defined by \eqref{eq:three_dw_kappa} with $\kappa(x)$ given by \eqref{eq:explicit_dw_func} and $\d = 2$.}
\label{fig:3_dw_kappa}
\end{figure}
\begin{figure} 
\includegraphics[scale=.6]{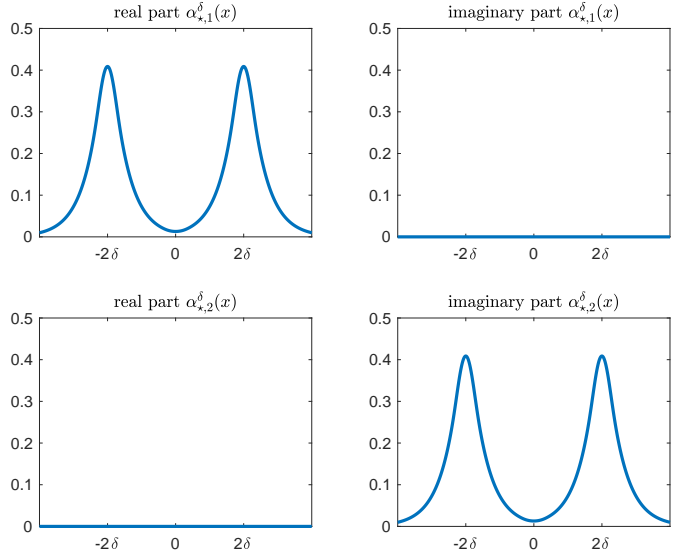}
\caption{Plot of the zero mode $\alpha^\d_\star(x)$ of $\mathcal{D}_{\kappa^\d}$ when $\kappa^\d(x)$ is given by \eqref{eq:three_dw_kappa}, given analytically by \eqref{eq:n_dw_zero_mode}.} 
\label{fig:3_dw_zero_mode}
\end{figure}

Introduce the following states, generated by the single domain wall zero mode:
\begin{equation} \label{eq:shifted_zero_modes}
	\alpha_\star^R := \alpha_\star(x - 2\d), \quad \alpha_\star^0 := \overline{\alpha_\star(x)}, \quad  \alpha_\star^L := \alpha_\star(x + 2\d),
\end{equation}
then Theorem \ref{th:main_theorem} generalizes to this case as follows: 
\begin{theorem} \label{th:3_dw}
Let Assumption \ref{as:assumption_on_D} hold and pick any compact interval $[-K,K] \subset (-\kappa_\infty,\kappa_\infty)$. Then, there is a constant $\d(K) \geq 2$ such that for all $\d > \d(K)$, \emph{in addition to the exact zero mode \eqref{eq:n_dw_zero_mode}}, the operator $\mathcal{D}_{\kappa^\d}$ has two near-zero eigenvalues $E_\pm(\d)$ in the interval $[-K,K]$, which satisfy: 
\begin{equation}
	E_\pm = \pm 2 \sqrt{2} \gamma^2 e^{- 2 \inty{0}{\d}{\kappa(y)}{y} } + O(e^{- 4 \kappa_\infty \d}).
\end{equation}
Their associated (normalized) eigenfunctions, which we denote $\alpha_\pm(x)$, may be written as approximate linear combinations of $\alpha_\star^\+(x), \alpha_\star^\-(x), \alpha_\star^0(x)$ as defined by \eqref{eq:shifted_zero_modes}:
\begin{equation} \label{eq:3_dw_modes}
\begin{split}
	&\alpha_+(x) = \frac{1}{2} \left( \alpha_\star^\+(x) + \sqrt{2} i \alpha_\star^0(x) - \alpha_\star^\-(x) \right) + O(e^{- 2 \kappa_\infty \d}) \\
	&\alpha_-(x) = \frac{1}{2} \left( \alpha_\star^\+(x) - \sqrt{2} i \alpha_\star^0(x) - \alpha_\star^\-(x) \right) + O(e^{- 2 \kappa_\infty \d}).
\end{split}
\end{equation}
\end{theorem}
For a numerical computation of the modes \eqref{eq:3_dw_modes}, see Figure \ref{fig:three_domain_walls_modes}. 
\begin{figure}
\includegraphics[scale=.6]{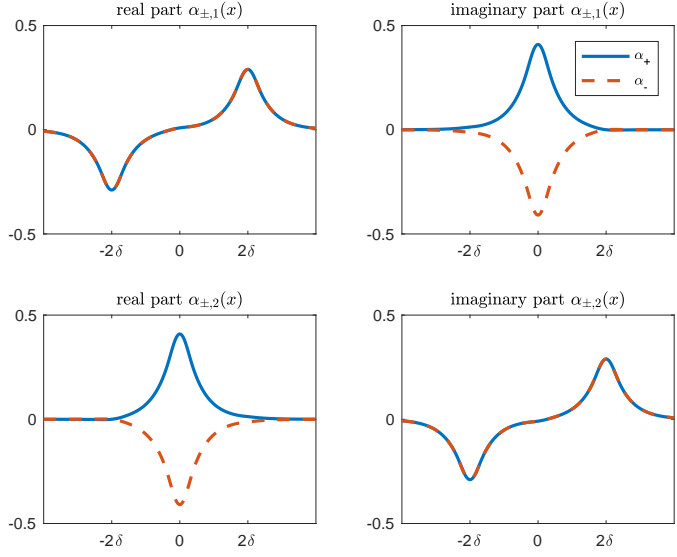}
\caption{Numerical computation of the two near-zero modes \eqref{eq:3_dw_modes} of the operator $\mathcal{D}_{\kappa^\d}$, for $\kappa^\d$ defined by \eqref{eq:three_dw_kappa}. $\mathcal{D}_{\kappa^\d}$ always has an exact zero mode given by \eqref{eq:n_dw_zero_mode}, see Figure \ref{fig:3_dw_zero_mode}.}
\label{fig:three_domain_walls_modes}
\end{figure}

\subsection{Proof of Theorem \ref{th:3_dw} on near-zero energy bound states of the three domain wall operator} \label{sec:3_dom_wall_proofs}
We now give the proof of Theorem \ref{th:3_dw}. Just as in Section \ref{sec:strategy_of_proof}, we seek a solution of \eqref{eq:eigenvalue_problem} as a linear combination of shifted zero modes of the one domain wall operator \eqref{eq:shifted_zero_modes} plus a corrector function:
\begin{equation}\label{3-ansatz}
	\alpha(x) = \sum_{j = \+, 0, \-} b^j \alpha^j_{\star}(x) + \eta(x),\quad {\rm where}\quad \eta\in {\rm span}\{\alpha_{\star}^R, \alpha_{\star}^0, \alpha_{\star}^L\}^\perp
\end{equation}

Unlike the the 2- domain wall case, the  decomposition 
\[\mathcal{H}={\rm span}\{\alpha_{\star}^R\} \oplus {\rm span} \{\alpha_{\star}^0\} \oplus {\rm span} \{ \alpha_{\star}^L\}\oplus {\rm span}\{\alpha_{\star}^R, \alpha_{\star}^0, \alpha_{\star}^L\}^\perp\] 
is not orthogonal. On the other hand, we have
\begin{lemma}\label{not-orthog}
\begin{equation}
	\ip{\alpha_\star^j}{\alpha_\star^j}_\mathcal{H} = 1,\qquad j\in\{\+,0,\-\}\ ,
\end{equation}
\begin{equation} \label{eq:bound_1_again}
	\ip{\alpha_\star^\-}{\alpha_\star^0}_\mathcal{H} = \ip{\alpha_\star^0}{\alpha_\star^\+}_\mathcal{H} = 0\ ,
\end{equation}
but
\begin{equation} \label{eq:bound_1_prime}
	0\ <\ \Bigl\lvert\ \ip{\alpha_\star^\-}{\alpha_\star^\+}_\mathcal{H}\ \Bigr\rvert \leq C e^{- 4 \kappa_\infty \d} \ .
\end{equation}
\end{lemma}
\begin{proof} 
\eqref{eq:bound_1_again} follows by an identical argument to that given in the proof of Lemma \ref{orthog-sym}. The proof of \eqref{eq:bound_1_prime} is as follows. Using the bound \eqref{eq:alpha_for_x_large}:
\begin{equation}
\begin{split}
	\ip{ \alpha_\star^\- }{ \alpha_\star^\+ }_{\mathcal{H}} &= \inty{-\infty}{\infty}{ \gamma^2 \ip{ \begin{pmatrix} 1 \\ i \end{pmatrix} }{ \begin{pmatrix} 1 \\ i \end{pmatrix} } e^{- \inty{0}{x - 2 \d}{\kappa(y)}{y}} e^{- \inty{0}{x + 2 \d}{\kappa(y)}{y}} }{x}	\\ 
	&\leq 2 C^2 \gamma^2 \inty{-\infty}{\infty}{ e^{ - \kappa_\infty | x - 2 \d | } e^{ - \kappa_\infty | x + 2 \d | } }{x}.
\end{split}
\end{equation}
Splitting the integral over the real line into an integral over three disjoint intervals:
\begin{equation}
\begin{split}
	&2 C^2 \gamma^2 \inty{-\infty}{\infty}{ e^{ - \kappa_\infty | x - 2 \d | } e^{ - \kappa_\infty | x + 2 \d | } }{x}	\\
	&= 2 C^2 \gamma^2 \left[ \inty{-\infty}{- 2\d}{ e^{ \kappa_\infty ( x - 2 \d ) } e^{ \kappa_\infty ( x + 2 \d ) } }{x} \right.	\\
	&\hphantom{=}\left. + \inty{- 2\d}{2 \d}{ e^{ \kappa_\infty ( x - 2 \d ) } e^{ - \kappa_\infty ( x + 2 \d ) } }{x} + \inty{2\d}{\infty}{ e^{ - \kappa_\infty ( x - 2 \d ) } e^{ - \kappa_\infty ( x + 2 \d ) } }{x}  \right]	\\ 
	&= 2 C^2 \gamma^2 \left[ \inty{-\infty}{- 2\d}{ e^{ 2 \kappa_\infty x} }{x} + \inty{- 2\d}{2 \d}{ e^{ - 4 \kappa_\infty \d } }{x} + \inty{2\d}{\infty}{ e^{ - 2 \kappa_\infty x } }{x} \right]	\\ 
	&\leq 2 C^2 \left( \frac{1}{2 \kappa_\infty} + 4 \d + \frac{1}{2 \kappa_\infty} \right) e^{- 4 \kappa_\infty \d}.
\end{split}
\end{equation}

\end{proof} 

We want to argue, as in the 2- domain wall case, that the eigenvalue problem is equivalent to the set of four coupled equations obtained by substitution of \eqref{3-ansatz} into the equation $\mathcal{D}_{\kappa^\d}\alpha=E\alpha$ and setting to zero, individually, the projections onto
 ${\rm span}\{\alpha_\star^\+\}$, ${\rm span}\{\alpha_\star^\-\}$,  ${\rm span}\{\alpha_\star^0\}$ 
 and $\Hrol \equiv {\rm span}\{\alpha_\star^\+,\alpha_\star^0, \alpha_\star^\-\}^\perp$. For large $\delta$, this is justified by the following lemma. First define $\Prol$ to be the orthogonal projection of $\mathcal{H}$
  onto the subspace $\Hrol$. 
 \begin{lemma}\label{near-orthog}
 There exists $\delta_1>2$ such that for all $\delta>\delta_1$: If $f \in \mathcal{H}$, then: 
 \begin{equation}
 f=0\quad  \textrm{if and only if}\quad \Prol f\ =\ 0\quad {\rm and}\quad \ip{\alpha_\star^j}{f}=0\quad
  \textrm{for}\ j=\+,0,\-.
  \end{equation}
 \end{lemma}

Now proceeding in a manner analogous to the 2- domain wall case, we have, for $E$ sufficiently small, that $E$ is an eigenvalue of \eqref{eq:eigenvalue_problem} if and only if there exists a nontrivial triple: $(b^\+ ,b^0,b^\- )$ which satisfies:
\begin{equation}\label{eq:reduced_equation_for_bs_again}
	\sum_{j \in \+,0,\-} M^{i j}(\d,E)  b^j = 0\ , \qquad\text{$i \in \{\+,0,\-\}$, } 
	\end{equation}
	where for $i, j=R,0,L$:
	\begin{equation}
	M^{ij}(\d,E) := \ip{ \alpha_\star^i }{ \mathcal{D}_{\kappa^\d} \alpha_\star^j }_\mathcal{H} - E \ip{ \alpha_\star^i }{ \alpha_\star^j } - \ip{ \Prol \mathcal{D}_{\kappa^\d} \alpha_\star^i }{ (\mathcal{D}_{\kappa^\d} - E)^{-1} \Prol \mathcal{D}_{\kappa^\d} \alpha_\star^j }_\mathcal{H}\ .
\end{equation}
The analogous statement to Corollary \ref{cor:corollary_on_DL} in this setting is now clear:
\begin{corollary} \label{cor:corollary_on_DL_2}
Fix $E$ such that $|E| \leq K$. There exists $\delta_1(K) \geq 2$ sufficiently large such that for all $\d > \delta_1(K)$ 
%
%
\begin{equation}
	\text{$E$ is an eigenvalue of $\mathcal{D}_{\kappa^\d}$ if and only if $\det M^{ij}(\d,E) = 0$.}
\end{equation}
\end{corollary}
In analyzing the set of roots of $\det M^{ij}(\d,E)$, we will make use of the following result which is analogous to Proposition \ref{lem:bounds_on_terms_in_M}
%
%
\begin{proposition} \label{prop:bounds_on_terms_in_M_again}
Let $M(\d,E)$ be as in \eqref{eq:reduced_equation_for_bs_again}. The entries of $E\mapsto M(\d,E)$ are analytic in a neighborhood of $E=0$. Furthermore, the matrix  $M(\d,E)$ may be expanded as:
\begin{equation} \label{eq:M_ij_with_bounds_again}
	M(\d,E) = \begin{pmatrix} - E & - 2 i \gamma^2 e^{- 2 \inty{0}{\d}{\kappa(y)}{y}} & 0 \\ 2 i \gamma^2 e^{- 2 \inty{0}{\d}{\kappa(y)}{y}} & - E & - 2 i \gamma^2 e^{- 2 \inty{0}{\d}{\kappa(y)}{y} } \\ 0 & 2 i \gamma^2 e^{- 2 \inty{0}{\d}{\kappa(y)}{y}} & - E  \end{pmatrix} + M_1(\d,E)\ .
\end{equation}
Here, each entry of $M_1(\d,E)$ satisfies the bound: $| M^{ij}_1(\d,E) | \leq C e^{- 4 \kappa_\infty \d} + C E e^{- 4 \kappa_\infty \d}$ for some constant $C > 0$ independent of $\d, E$. 
\end{proposition}
The proof of Proposition \ref{prop:bounds_on_terms_in_M_again} is given in Appendix \ref{sec:proof_3dw_lemma}. We are now in a position to prove Theorem \ref{th:3_dw}. Just as in the proof of Theorem \ref{th:main_theorem}, we compute the determinant of $M^{ij}$ \eqref{eq:reduced_equation_for_bs_again} making use of \eqref{eq:M_ij_with_bounds_again}:  
\begin{equation}
	\text{det } M^{ij}(\d,E) = E \left( 2 \left( 2 \gamma^2 e^{- 2 \inty{0}{\d}{\kappa(y)}{y} } \right)^2 - E^2 \right) + g(\d,E),
\end{equation}
where $g(\d,E)$ is analytic in $E$ and satisfies the bound:
\begin{equation}
	| g(\d,E) | \leq C ( e^{- 8 \kappa_\infty \d} + E e^{- 6 \kappa_\infty \d} + E^2 e^{- 4 \kappa_\infty \d} + E^3 e^{- 8 \kappa_\infty \d} )
\end{equation}
for some constant $C > 0$ independent of $\d, E$. Via arguments analogous to those given in Section \ref{sec:strategy_of_proof} below Proposition \ref{lem:bounds_on_terms_in_M} to prove Theorem \ref{th:main_theorem}, we may now conclude by an application of Rouch\'{e}'s theorem that $\mathcal{D}_{\kappa^\d}$ has precisely \emph{three} near-zero eigenvalues $E_+(\d), E_0(\d)$ and $E_-(\d)$ which satisfy: 
\begin{equation} \label{eq:3_dw_evalues}
\begin{split}
	&E_+(\delta) = 2 \sqrt{2} \gamma^2 e^{- 2 \inty{0}{\d}{\kappa(y)}{y} } + O(e^{- 4 \kappa_\infty \d})	\\
	&E_0(\d) = O(e^{- 4 \kappa_\infty \d})	\\
	&E_-(\d) = - 2 \sqrt{2} \gamma^2 e^{- 2 \inty{0}{\d}{\kappa(y)}{y} } + O(e^{- 4 \kappa_\infty \d}).
\end{split}
\end{equation}
Recall that the three-domain wall operator $\mathcal{D}_{\kappa^\d}$ has an explicit zero mode (eigenfunction with eigenvalue equal to zero) displayed in \eqref{eq:n_dw_zero_mode}. By a further application of Rouch\'{e}'s theorem, we find that $E_0(\d)$ is the unique eigenvalue of $\mathcal{D}_{\kappa^\d}$ in an appropriate ball of radius $c e^{- 4 \kappa_\infty \d}$ centered at $E = 0$ for some constant $c > 0$, and hence $E_0(\d)$ must be precisely zero:
\begin{equation}
	E_0(\d) = 0.
\end{equation}
Furthermore, the associated eigenfunction of $E_0(\d)$ must be given by \eqref{eq:n_dw_zero_mode}.
%
%

Substituting expressions \eqref{eq:3_dw_evalues} for $E_+(\delta)$ and $E_-(\delta)$ into \eqref{eq:reduced_equation_for_bs_again}, we obtain expansions of the associated eigenfunctions of these eigenvalues: 
\begin{equation}
\begin{split}
	&\begin{pmatrix} b^\+ \\ b^0 \\ b^\- \end{pmatrix}_+ = \frac{1}{4} \begin{pmatrix} 1 \\ \sqrt{2} i \\ - 1 \end{pmatrix} + O(e^{- 2 \kappa_\infty \d} ) 	\\
	&\begin{pmatrix} b^\+ \\ b^0 \\ b^\- \end{pmatrix}_- = \frac{1}{4} \begin{pmatrix} 1 \\ - \sqrt{2} i \\ -1 \end{pmatrix} + O(e^{- 2 \kappa_\infty \d} ).
\end{split}
\end{equation}
This completes the proof of Theorem \ref{th:3_dw}. 
%


\appendix
\section{For  $\kappa(x) = \tanh(x)$, point spectrum of $\mathcal{D}_\kappa$ is equal to $\{0\}$ } \label{sec:tanh_unique_evalue}
Let $\mathcal{D}_{\tanh}$ be defined by \eqref{eq:single_dom_wall_operator} with $\kappa(x) = \tanh(x)$:
\begin{equation}
	\mathcal{D}_{\tanh} := i \sigma_3 \de_x + \tanh(x) \sigma_1.
\end{equation}
It is clear that $\tanh(x)$ satisfies conditions \eqref{eq:weaker_1}-\eqref{eq:weaker_2} with $\kappa_\infty = 1$ and hence has continuous spectrum $(-\infty,-1] \cup [1,\infty)$. We will show that the only possible eigenfunction of $\mathcal{D}_{\tanh}$ with eigenvalue in the gap $(-1,1)$ is the exact zero mode \eqref{eq:zero_mode}. Ideas related to the argument given here are known as ``0-dimensional supersymmetry'' (see Chapter 6.3 of \cite{CyconFroeseKirschSimon}, for example).

Let $E$ denote any eigenvalue of $\mathcal{D}_{\tanh}$ in the interval $(-1,1)$. Let $\alpha = (\alpha_1,\alpha_2)^T$ denote the associated eigenfunction of such an eigenvalue, so that:
\begin{equation} \label{eq:dirac}
	\begin{pmatrix} i \de_x & \tanh(x) \\ \tanh(x) & - i \de_x \end{pmatrix} \begin{pmatrix} \alpha_1(x) \\ \alpha_2(x) \end{pmatrix} = E \begin{pmatrix} \alpha_1(x) \\ \alpha_2(x) \end{pmatrix}.
\end{equation}
Applying the unitary matrix:
\begin{equation} \label{eq:U}
	U := \frac{1}{\sqrt{2}} \begin{pmatrix} 1 & i \\ 1 & - i \end{pmatrix} 
\end{equation}
to both sides of \eqref{eq:dirac} we obtain an equivalent equation for $\beta = (\beta_1,\beta_2)^T := U \alpha$:
\begin{equation} \label{eq:dirac_again}
	\begin{pmatrix} 0 & i \de_x + i \tanh(x) \\ i \de_x - i \tanh(x) & 0 \end{pmatrix} \begin{pmatrix} \beta_1(x) \\ \beta_2(x) \end{pmatrix} = E \begin{pmatrix} \beta_1(x) \\ \beta_2(x) \end{pmatrix}.
\end{equation}
By squaring the operator on the left-hand side and using standard trigonometric identities, we see that $(\beta_1, \beta_2)^T$ must satisfy:
\begin{equation} \label{eq:dirac_squared}
	\begin{pmatrix} - \de_x^2 + 1 & 0 \\ 0 & - \de_x^2 + (1 - 2 \sech^2(x)) \end{pmatrix} \begin{pmatrix} \beta_1(x) \\ \beta_2(x) \end{pmatrix} = E^2 \begin{pmatrix} \beta_1(x) \\ \beta_2(x) \end{pmatrix}.
\end{equation}
We now claim that $\beta_1(x)$ must equal $0$. If not, \eqref{eq:dirac_squared} implies that $\beta_1(x)$ would be an eigenfunction of the operator $- \de_x^2 + 1$ with eigenvalue $E^2$, which by assumption on $E$ lies in the interval $[0,1)$. But this is impossible since the spectrum of $-\de_x^2 + 1$ is precisely $[1,\infty)$. 

Substituting $\beta_1 = 0$ into \eqref{eq:dirac_again}, we see that $\beta_2$ must satisfy: 
\begin{equation}
	( \de_x + \tanh(x) ) \beta_2(x) = 0, \quad E \beta_2(x) = 0.
\end{equation}
For a non-trivial solution we must have that $E = 0$ and:
\begin{equation}
	\beta_2(x) = C e^{- \inty{0}{x}{ \tanh(y) }{y} }
\end{equation}
for some non-zero complex constant $C$. Inverting the change of basis defined by $U$ \eqref{eq:U} and choosing the complex constant to ensure normalization we see that the only possible eigenfunction of $\mathcal{D}_{\tanh}$ with eigenvalue in the interval $(-1,1)$ is the zero mode \eqref{eq:zero_mode}. We remark finally that using standard identities, when $\kappa(x) = \tanh(x)$ the zero mode may be written in closed form:
\begin{equation}
	\alpha_\star(x) = \frac{1}{\sqrt{2} \pi} \begin{pmatrix} 1 \\ i \end{pmatrix} \sech(x)	
\end{equation}

\section{Relationship of $\mathcal{D}_{\kappa^\d}$ to Witten Laplacians} \label{ap:remark_on_schrodinger_link}
Let $U$ be the unitary matrix \eqref{eq:U} and $\mathcal{D}_{\kappa^\d}$ be the Dirac operator defined by \eqref{eq:def_dirac}. Conjugating $\mathcal{D}_{\kappa^\d}$ by $U$ yields the operator:
\begin{equation} \label{eq:onee}
    U \mathcal{D}_{\kappa^\d} U^* = \begin{pmatrix} 0 & i \de_x + i \kappa^{\d}(x) \\ i \de_x - i \kappa^\d(x) & 0 \end{pmatrix}\ =\ i\partial_x\sigma_1 -\kappa^\d(x)\sigma_2.
\end{equation}
Squaring this operator we derive 
\begin{equation} \label{eq:twoo}
    \left( U \mathcal{D}_{\kappa^\d} U^* \right)^2 = U \mathcal{D}_{\kappa^\d}^2 U^* = \begin{pmatrix} - \de_x^2 + ( \kappa^\d(x) )^2 - (\kappa^\d)'(x) & 0 \\ 0 & - \de_x^2 + ( \kappa^\d(x) )^2 + (\kappa^\d)'(x) \end{pmatrix}.
\end{equation}
Operators with the form
\begin{equation}
    - \de_x^2 + | f'(x) |^2 - f''(x),
\end{equation}
where $f(x)$ is a real function, are known as ``Witten Laplacians'' in one dimension \cite{1982Witten,CyconFroeseKirschSimon,HelfferNier}. The calculation \eqref{eq:onee}-\eqref{eq:twoo} shows that the square of $\mathcal{D}_{\kappa^\d}$ is unitarily equivalent to a diagonal operator whose diagonal entries are Witten Laplacians with $f'(x) = \pm \kappa^{\d}(x)$. 

The semi-classical analysis of ``low-lying'' (near zero) eigenvalues in the limit $\d \rightarrow \infty$ of such operators is well-established (see, for example, \cite{1980Harrell,1984Simon,1984HelfferRobert,1984HelfferSjostrand,1985HelfferSjostrand,1985KirschSimon,1986Nakamura,1986Martinez,1987KirschSimon,1988MartinezRouleux,1988GerardGrigis,Dimassi-Sjoestrand:99}) and could be used to provide an alternative proof of our our main results by the following argument. Note first that the operator \eqref{eq:onee} anti-commutes with the Pauli matrix $\sigma_3$ and hence every non-zero eigenvalue of the squared Dirac operator \eqref{eq:twoo} is two-fold degenerate. Second, note that given a basis of the degenerate two-dimensional subspace of \eqref{eq:twoo}, the correct eigenvalues and eigenfunctions of the original Dirac operator can be recovered by diagonalizing \eqref{eq:onee} on this degenerate subspace. Our results then follow from recognizing that the degenerate subspace of \eqref{eq:twoo} is spanned by the eigenfunctions of the Schr\"odinger operators on the diagonal multiplied by the standard basis functions in $\mathbb{C}^2$. Since these eigenfunctions and associated eigenvalues can be computed using semiclassical theory for Schr\"odinger operators, our results would follow. We prefer working with the Dirac operators directly instead. 

\section{Proof of Theorem \ref{th:main_theorem} (proofs of key Propositions)} \label{sec:proofs}
\subsection{Proof of Proposition \ref{prop:app_zm}} \label{sec:proofs_0}
Let $\alpha_\star^\+(x)$ and $\alpha_\star^\-(x)$ denote the shifted one domain wall zero modes defined by \eqref{eq:def_alpha_pm}, and $\mathcal{D}_{\kappa^\d}$ be the two domain wall Dirac operator defined by \eqref{eq:kappa_L}. Without loss of generality since estimating $\mathcal{D}_{\kappa^\d} \alpha_\star^\-$ is similar, we prove the estimate for $\mathcal{D}_{\kappa^\d} \alpha_\star^\+$ only. By definition:
\begin{equation}
	\| \mathcal{D}_{\kappa^\d} \alpha_\star^\+ \|_\mathcal{H}^2 = \inty{-\infty}{\infty}{ \left| \mathcal{D}_{\kappa^\d} \alpha^\+_\star(x) \right|^2 }{x}.
\end{equation}
Splitting the integral on the right-hand side into integrals over disjoint intervals gives:
\begin{equation}
	\inty{-\infty}{\infty}{ \left| \mathcal{D}_{\kappa^\d} \alpha^\+_\star(x) \right|^2 }{x} = \inty{-\infty}{- \d + 1}{ \left| \mathcal{D}_{\kappa^\d} \alpha^\+_\star(x) \right|^2 }{x} + \inty{- \d + 1}{\infty}{ \left| \mathcal{D}_{\kappa^\d} \alpha^\+_\star(x) \right|^2 }{x}.
\end{equation}
Acting on functions supported on the interval $[- \d + 1,\infty)$, the operators $\mathcal{D}_{\kappa^\d}$ and $\mathcal{D}^\+_{\kappa}$ are equal. Since $\mathcal{D}_{\kappa^\d} \alpha^\+_\star = 0$, the second integral is then equal to zero. The first integral may be bounded using boundedness of $\kappa$ and exponential decay of $\alpha^\+_\star$ \eqref{eq:alpha_for_x_large} yielding the estimate:
\begin{equation} \label{eq:estimate}
	\| \mathcal{D}_{\kappa^\d} \alpha_\star^\+ \|_\mathcal{H}^2 \leq C' \inty{-\infty}{-\d + 1}{ e^{- 2 \kappa_\infty | x - \d |} }{x}
\end{equation}
where $C' > 0$ is a constant depending only on $\kappa$. The integral on the right-hand side may be evaluated as follows:
\begin{equation}
\begin{split}
	&\inty{-\infty}{-\d + 1}{ e^{- 2 \kappa_\infty | x - \d |} }{x} = \inty{-\infty}{-\d + 1}{ e^{- 2 \kappa_\infty ( \d - x ) } }{x} \\ 
	&= e^{- 2 \kappa_\infty \d} \inty{-\infty}{-\d + 1}{ e^{ 2 \kappa_\infty x } }{x} = \frac{ e^{- 2 \kappa_\infty} }{ 2 \kappa_\infty } e^{- 4 \kappa_\infty \d}.
\end{split}
\end{equation}
Substituting into \eqref{eq:estimate} and taking the square root now yields the estimate: 
\begin{equation}
	\| \mathcal{D}_{\kappa^\d} \alpha_\star^\+ \|_\mathcal{H} \leq C e^{- 2 \kappa_\infty \d}
\end{equation}
where $C > 0$ is a constant depending only on the function $\kappa$. 

\subsection{Proof of Proposition \ref{prop:prop_on_DL}} \label{sec:proofs_1}
Let $f \in \mathcal{H}$ be such that: $\ip{\alpha_{\star}^R}{f}_\mathcal{H} = \ip{\alpha_{\star}^L}{f}_\mathcal{H} = 0$. We will prove assertion \eqref{eq:statement_lemma_on_DL} by bounding $\mathcal{D}_{\kappa^\delta} f$ from below on disjoint subsets of the real line and then summing these estimates
To this end, we introduce the partition of unity:
\begin{equation}
	1 = (\theta^\-(x))^2 + (\theta^0(x))^2 + (\theta^\+(x))^2
\end{equation}
where the functions $\theta^\+, \theta^\-, \theta^0$ are assumed to be smooth and to satisfy:
\begin{equation}
	\theta^\-(x) = \begin{cases} 1 &\text{ for } x \leq - \frac{\d}{2} \\ 0 &\text{ for } x \geq - \frac{\d}{4} \end{cases}
\end{equation}
\begin{equation}
	\theta^0(x) = \begin{cases} 1 &\text{ for } - \frac{\d}{4} \leq x \leq \frac{\d}{4} \\ 0 &\text{ for } x \leq - \frac{\d}{2} \text{ or } x \geq \frac{\d}{2} \end{cases}
\end{equation}
\begin{equation}
	\theta^\+(x) = \begin{cases} 1 &\text{ for } x \geq \frac{\d}{2} \\ 0 &\text{ for } x \leq \frac{\d}{4} \end{cases}.
\end{equation}
Using the partition of unity, we have that: 
\begin{equation} \label{eq:using_partition}
\begin{split}
	\| \mathcal{D}_{\kappa^\d} f \|_\mathcal{H}^2 &= \sum_{j = 0,\+,\-} \inty{\field{R}}{}{ (\theta^j(x))^2 | \mathcal{D}_{\kappa^\d} f(x) |^2 }{x} \\
	&= \sum_{j = 0,\+,\-} \| \theta^j \mathcal{D}_{\kappa_\d} f \|^2_{\mathcal{H}}. 
\end{split}
\end{equation}
Before continuing, we note two consequences of the definitions of the $\theta^j(x), j \in \{\+,\-,0\}$. Recall that $\d \geq 2$ and will later be taken as large as necessary. First, for each positive integer $n \geq 1$: 
\begin{equation} \label{eq:derivatives_theta}
	\quad \sup_{x \in \field{R}} | \de_x^n \theta^j(x) | \leq \frac{C}{\d^n}, \quad j \in \{\+,\-,0\}
\end{equation}
for positive constants $C > 0$ which are independent of $\d$. Second, if $\d \geq 2$, then:
\begin{equation}
	- \d + 1 \leq - \frac{\d}{2} < - \frac{\d}{4}. 
\end{equation} 
If follows that for $\d \geq 2$:
\begin{equation} \label{eq:kappa_on_supports}
\begin{split}
	&\text{for $x \in \text{supp} [\theta^\+]$}, \quad\kappa^\d(x) = \kappa(x - \d) 	\\
	&\text{for $x \in \text{supp} [\theta^0]$}, \quad\kappa^\d(x) = - \kappa_\infty 	\\
	&\text{for $x \in \text{supp} [\theta^\-]$}, \quad\kappa^\d(x) = - \kappa(x + \d).
\end{split}
\end{equation} 
\underline{We assume at this point that $\d \geq 2$ so that (\ref{eq:kappa_on_supports}) holds.}

We now demonstrate how to bound each term in \eqref{eq:using_partition} from below. By a trivial re-arrangement we have that: 
\begin{equation} \label{eq:commutator}
	\theta^j(x) \mathcal{D}_{\kappa^\d} = \mathcal{D}_{\kappa^\d} \theta^j(x) + [\theta^j(x),\mathcal{D}_{\kappa^\d}] = \mathcal{D}_{\kappa^\d} \theta^j(x) - i \de_x \theta^j(x) \sigma_3.
\end{equation} 
Hence:
\begin{equation}
	\| \theta^j \mathcal{D}_{\kappa^\d} f \|^2_{\mathcal{H}} = \| \mathcal{D}_{\kappa^\d} \theta^j f - i \sigma_3 \de_x \theta^j f \|^2_{\mathcal{H}}, \quad j \in \{0,\+,\-\}.
\end{equation}
We then proceed as follows:
\begin{align} \nonumber
	&\| \mathcal{D}_{\kappa^\d} \theta^j f - i \sigma_3 \de_x \theta^j f \|^2_{\mathcal{H}}  &	\\
	\nonumber &= \| \mathcal{D}_{\kappa^\d} \theta^j f \|^2_{\mathcal{H}} + \| \de_x \theta^j f \|^2_{\mathcal{H}} + 2 \text{Re} \ip{i \sigma_3 \de_x \theta^j f}{ \mathcal{D}_{\kappa^\d} \theta^j f}_{\mathcal{H}} &	\\
	\nonumber &\geq \| \mathcal{D}_{\kappa^\d} \theta^j f \|^2_{\mathcal{H}} + \| \de_x \theta^j f \|^2_{\mathcal{H}} - 2 \| \sigma_3 \de_x \theta^j f \|^2_{\mathcal{H}} \| \mathcal{D}_{\kappa^\d} \theta^j f \|^2_{\mathcal{H}} & \text{(Cauchy-Schwarz inequality)}	\\
	\nonumber &\geq (1 - \epsilon^2) \| \mathcal{D}_{\kappa^\d} \theta^j f \|^2_{\mathcal{H}} + (1 - \epsilon^{-2}) \| \de_x \theta^j f \|^2_\mathcal{H} & \text{(Young's inequality)}
\end{align}
for any positive $\epsilon > 0$. Taking $\epsilon$ sufficiently small so that $(1 - \epsilon^{-2}) \leq \epsilon^{-2}$ and using (\ref{eq:derivatives_theta}), we have that each term in \eqref{eq:using_partition} satisfies:
\begin{equation} \label{eq:commuted_theta}
\begin{split}
	\| \theta^j \mathcal{D}_{\kappa^\d} f \|^2_{\mathcal{H}} \geq ( 1 - \epsilon^2 ) \| \mathcal{D}_{\kappa^\d} \theta^j f \|_{\mathcal{H}}^2 - \epsilon^{-2} \frac{C_1}{\d^2} \| f \|^2_{\mathcal{H}}, \quad j \in \{0,\+,\-\}.
\end{split}
\end{equation}
for some constant $C_1 > 0$ which is independent of $\d$. We now study the terms:
\begin{equation}
	\| \mathcal{D}_{\kappa^\d} \theta^j f \|^2_{\mathcal{H}}, \quad j \in \{0,\+,\-\}.
\end{equation}
First, we consider $j = \+$. On the support of $\theta^\+$, $\kappa^\d(x) = \kappa(x - \d)$ \eqref{eq:kappa_on_supports}, hence:
\begin{equation} \label{eq:theta_+}
	\| \mathcal{D}_{\kappa^\d} \theta^\+ f \|^2_{\mathcal{H}} = \| \mathcal{D}^R_{\kappa} \theta^\+ f \|_{\mathcal{H}}^2,  \quad j \in \{0,\+,\-\}.
\end{equation}
where $\mathcal{D}^R_{\kappa} = i \sigma_3 \de_x + \kappa(x - \d)$ is the `shifted' one domain wall operator \eqref{eq:shifted_operators}. In order to bound \eqref{eq:theta_+} from below we use the fact that $\ip{\alpha^\+_\star}{f}_\mathcal{H} = 0$. We first show that this implies that $\ip{\alpha_\star^R}{\theta^\+ f}_{\mathcal{H}}$ is exponentially small in $\delta$: 
\begin{align*}
	&\ip{\alpha^\+_\star}{f}_\mathcal{H} = 0	& \\
	&\iff \ip{\alpha^\+_\star}{\theta^\+ f}_\mathcal{H} = - \ip{\alpha^\+_\star}{ (1 - \theta^\+) f }_\mathcal{H} &	\\
	&\implies \ip{\alpha^\+_\star}{ \theta^\+ f }_\mathcal{H} \leq \| \alpha^\+_\star (1 - \theta^\+ ) \|_\mathcal{H} \| f \|_\mathcal{H} &\text{(Cauchy-Schwarz)}
\end{align*}
but:
\begin{align*}
	&\| (1 - \theta^\+ ) \alpha^\+_\star \|_\mathcal{H}^2 = \inty{}{}{ | (1 - \theta^\+(x)) \alpha_\star(x - \d) |^2 }{x} &	\\
	&\leq \inty{-\infty}{\frac{\d}{2}}{ | \alpha_\star(x - \d) |^2 }{x} &\text{(Supp $(1 - \theta^\+ ) = [-\infty,\d/2]$)} 	\\ 
	&\leq C \inty{- \infty}{\frac{\d}{2}}{ e^{- 2 \kappa_\infty |x - \d|} }{x} = C \inty{- \infty}{\frac{\d}{2}}{ e^{2 \kappa_\infty (x - \d)} }{x} &\text{(Since $\d/2 \geq 1$, using (\ref{eq:alpha_for_x_large}))}	\\
	&= C e^{- 2 \kappa_\infty \d} \inty{-\infty}{\d/2}{ e^{2 \kappa_\infty x} }{x} \leq C e^{- \kappa_\infty \d} &
\end{align*}
where $C > 0$ is a constant which is independent of $\d$. We have therefore that:
\begin{equation} \label{eq:f_orthogonal}
	\ip{\alpha^\+_\star}{ \theta^\+ f }_\mathcal{H} \leq \| \alpha^\+_\star (1 - \theta^\+) \|_\mathcal{H} \| f \|_\mathcal{H} \leq C e^{- \kappa_\infty \d / 2} \| f \|_\mathcal{H}. 
\end{equation}
We now require the following corollary of Proposition \ref{prop:prop_on_D}:
\begin{corollary} \label{cor:corollary_to_lemma_on_D}
Let $f \in \mathcal{H}$. Then: 
\begin{equation}
	\| \mathcal{D}_\kappa f \|_\mathcal{H} \geq \kappa_\infty \left| \, \| f \|_\mathcal{H} - |\ip{\alpha_\star}{f}_\mathcal{H}| \, \right|.
\end{equation}
\end{corollary}
\begin{proof}
For any $f \in \mathcal{H}$, we may write $P^\perp_\star f \equiv f- \ip{\alpha_\star}{f}_\mathcal{H} \alpha_\star$. Since $\alpha_\star$ is a zero mode of $\mathcal{D}_\kappa$ and using Proposition \ref{prop:prop_on_D} we have:
\begin{equation}
	\| \mathcal{D}_\kappa f \|_\mathcal{H} = \| \mathcal{D}_\kappa P^\perp_\star f \|_\mathcal{H} \geq \kappa_\infty \| P^\perp_\star f \|_\mathcal{H}.
\end{equation}
Furthermore, 
\begin{equation}
	\| \mathcal{D}_\kappa f \|_\mathcal{H} \geq \kappa_\infty \| P_\star^\perp f \|_\mathcal{H} = \kappa_\infty \| f - \ip{\alpha_\star}{f}_\mathcal{H} \alpha_\star \|_\mathcal{H} \geq \kappa_\infty \left| \left( \| f \|_\mathcal{H} - |\ip{\alpha_\star}{f}_\mathcal{H}| \right) \right|.
\end{equation}
\end{proof}
It then follows from combining \eqref{eq:theta_+}, \eqref{eq:f_orthogonal}, and Corollary \ref{cor:corollary_to_lemma_on_D} that:
\begin{equation} \label{eq:+_part}
	\| \mathcal{D}_{\kappa^\d} \theta^\+ f \|^2_\mathcal{H} = \| \mathcal{D}^\+_{\kappa} \theta^\+ f \|^2_\mathcal{H} \geq \kappa_\infty^2 \left( 1 - C_2 e^{- \kappa_\infty \d / 2} \right)^2 \| \theta^\+ f \|^2_\mathcal{H}
\end{equation}
for some constant $C_2 > 0$ which is independent of $\d$. An identical argument shows that:
\begin{equation} \label{eq:-_part}
	\| \mathcal{D}_{\kappa^\d} \theta^\- f \|^2_\mathcal{H} \geq \kappa_\infty^2 \left( 1 - C_3 e^{- \kappa_\infty \d / 2} \right)^2 \| \theta^\- f \|^2_\mathcal{H},
\end{equation}
where $C_3 > 0$ is another constant which is independent of $\d$. Finally, we have that:
\begin{align}
\nonumber	&\| \mathcal{D}_{\kappa^\d} \theta^0 f \|^2_\mathcal{H} = \inty{}{}{ |(i \sigma_3 \de_x + \kappa^\d(x) )( \theta^0(x) f(x) ) |^2 }{x} & \\
\nonumber	&= \inty{}{}{ |(i \sigma_3 \de_x - \kappa_\infty \sigma_1 )( \theta^0(x) f(x) ) |^2 }{x} &\text{(Using (\ref{eq:kappa_on_supports}))}	\\
\label{eq:0_part} 	&= \ip{\theta^0 f}{ \left( i \sigma_3 \de_x - \kappa_\infty \sigma_1 \right)^2 \theta^0 f }_\mathcal{H} \geq \kappa_\infty^2 \| \theta^0 f \|^2_\mathcal{H}.
\end{align} 
Summing (\ref{eq:+_part}), (\ref{eq:-_part}), and (\ref{eq:0_part}) we see that for sufficiently large $\d \geq 2$: 
\begin{equation}
\begin{split}
	\sum_{j = 0,\+,\-} \| \mathcal{D}_{\kappa^\d} \theta^j f \|_\mathcal{H}^2 &\geq \kappa_\infty^2 \left( 1 - C_4 e^{- \kappa_\infty \d / 2} \right)^2 \sum_{j = 0,\+,\-} \| \theta^j f \|^2_\mathcal{H} 	\\
	&= \kappa_\infty^2 \left( 1 - C_4 e^{- \kappa_\infty \d / 2} \right)^2 \| f \|^2_\mathcal{H},
\end{split}
\end{equation}
where $C_4 := \max(C_2,C_3)$. Combining this with \eqref{eq:commuted_theta} we have that:
\begin{equation} \label{eq:est_on_D}
\begin{split}
	&\| \mathcal{D}_{\kappa^\d} f \|^2_\mathcal{H} \geq \left( (1 - \epsilon^2) \kappa_\infty^2 \left( 1 - C_4 e^{- \kappa_\infty \d/2} \right)^2 - \epsilon^{-2} \frac{C_1}{\d^2} \right) \| f \|^2_\mathcal{H} 	\\
\end{split}
\end{equation}
where the constants $C_1$ and $C_4$ do not depend on $\d$ or $\epsilon$. The proof of Proposition \ref{prop:prop_on_DL} is now as follows. Re-arranging \eqref{eq:est_on_D} we see that:
\begin{equation} \label{eq:est_on_D_2}
	\| \mathcal{D}_{\kappa^\d} f \|^2_\mathcal{H} \geq \kappa_\infty^2 \left( 1 - \epsilon^2 - 2 C_4 e^{- \kappa_\infty \d/2} - \epsilon^2 C_4^2 e^{- \kappa_\infty \d} - \frac{C_1 \epsilon^{-2} }{\kappa_\infty^2 \d^2} \right) \| f \|^2_\mathcal{H}.
\end{equation}
It now follows that by fixing $\epsilon$ sufficiently small and then $\d$ sufficiently large the constant multiplying $\| f \|^2_\mathcal{H}$ can be made arbitrarily close to $\kappa_\infty^2$ and hence, in particular, larger than:
\begin{equation}
	\left( \kappa_\infty - \frac{1}{2} (\kappa_\infty - K)^2 \right)^2 = \left( \frac{1}{2} ( \kappa_\infty + K ) \right)^2
\end{equation}
as required to prove the estimate \eqref{eq:statement_lemma_on_DL}. The second estimate \eqref{eq:second_part_lem_on_DL} follows from \eqref{eq:statement_lemma_on_DL} and an application of the triangle inequality. 

\subsection{Proof of Proposition \ref{lem:bounds_on_terms_in_M}} \label{sec:proofs_2}
Using self-adjointness of $\mathcal{D}_{\kappa^\d}$ and conjugate symmetry of $\ip{\cdot}{\cdot}_\mathcal{H}$, assertions \eqref{eq:M_ij_with_bounds} and \eqref{eq:resolvent_bounded} of Proposition \ref{lem:bounds_on_terms_in_M} follow immediately from:
$
	\ip{\alpha_\star^I}{\alpha_\star^J}_\mathcal{H} = \delta_{IJ},\ I,J\in\{R,L\}
$ (Lemma \ref{orthog-sym}) and  the following assertions:
\begin{equation}\label{eq:bd-1}
	\ip{\alpha_\star^\+}{\mathcal{D}_{\kappa^\d} \alpha_\star^\+}_\mathcal{H} = \ip{\alpha_\star^\-}{\mathcal{D}_{\kappa^\d} \alpha_\star^\-}_{\mathcal{H}} = 0,
\end{equation}
\begin{equation}
	\ip{\alpha_\star^\-}{\mathcal{D}_{\kappa^\d} \alpha_\star^\+}_\mathcal{H} = 2 i \gamma^2 e^{ - 2 \inty{0}{\delta}{ \kappa(y) }{y} },\label{eq:bd-2}
\end{equation}
\begin{equation} \label{eq:bound}
	\left| \ip{ \Prl \mathcal{D}_{\kappa^\d} \alpha_\star^i }{ (\mathcal{D}_{\kappa^\d} - E)^{-1} \Prl \mathcal{D}_{\kappa^\d} \alpha_\star^j }_\mathcal{H} \right| \leq C e^{- 4 \kappa_\infty \d}, \text{ for } i, j \in \{ \-, \+ \}.  
\end{equation}
We prove assertions \eqref{eq:bd-1}, \eqref{eq:bd-2}, and \eqref{eq:bound} in sections \ref{sec:2}-\ref{sec:4} below.

\subsubsection{Proof that $\ip{\alpha^\+_\star}{\mathcal{D}_{\kappa^\d} \alpha^\+_\star}_\mathcal{H} = \ip{\alpha^\-_\star}{\mathcal{D}_{\kappa^\d} \alpha^\-_\star}_\mathcal{H} = 0$} \label{sec:2}
Since showing $\ip{\alpha_\star^\-}{\mathcal{D}_{\kappa^\d} \alpha_\star^\-}_\mathcal{H} = 0$ is similar, we prove only that $\ip{\alpha_\star^\+}{\mathcal{D}_{\kappa^\d} \alpha_\star^\+}_\mathcal{H} = 0$ as follows:
\begin{equation}
\begin{split}
	&\ip{ \alpha_\star^\+ }{ \mathcal{D}_{\kappa^\d} \alpha_\star^\+ } = \inty{ -\infty }{ \infty }{ \ip{ \alpha_\star(x - \d) }{ \mathcal{D}_{\kappa^\d} \alpha_\star(x - \d) } }{x}	\\
	&= \inty{ -\infty }{ \infty }{ \ip{ \alpha_\star(x - \d) }{ \left( \sigma_3 i \de_x + \sigma_1 \kappa^\d(x) \right) \alpha_\star(x - \d) } }{x}.
\end{split}
\end{equation}
Explicit computation shows that:
\begin{equation} \label{eq:D_acting}
	(\sigma_3 i \de_x + \sigma_1 \kappa^\delta(x)) \begin{pmatrix} 1 \\ i \end{pmatrix} = \begin{pmatrix} i \\ 1 \end{pmatrix} (\de_x + \kappa^\delta(x))
\end{equation}
and hence:
\begin{equation}
\begin{split}
	\ip{ \alpha_\star^\+ }{ \mathcal{D}_{\kappa^\d} \alpha_\star^\+ } &= \gamma^2 \inty{- \infty}{\infty}{ \ip{ \begin{pmatrix} 1 \\ i \end{pmatrix} }{ \begin{pmatrix} i \\ 1 \end{pmatrix} } e^{- \inty{0}{x - \d}{ \kappa(y) }{y} } ( i \de_ x + \kappa^\d (x) ) e^{- \inty{0}{x - \d}{ \kappa(y) }{y} } }{x} \\
	&= 0,
\end{split}
\end{equation}
since $\ip{ \begin{pmatrix} 1 \\ i \end{pmatrix} }{ \begin{pmatrix} i \\ 1 \end{pmatrix} } = 0$. 

\subsubsection{Proof that $\ip{\alpha^\-_\star}{\mathcal{D}_{\kappa^\d}\alpha^\+_\star}_\mathcal{H} = 2 i \gamma^2 e^{- 2 \inty{0}{ \d}{\kappa(y)}{y}}$} \label{sec:3}
\begin{align*}
	\nonumber&\ip{\alpha_\star^\-}{\mathcal{D}_{\kappa^\d}\alpha_\star^\+}_\mathcal{H} = \inty{\field{R}}{}{ \ip{ \overline{ \alpha_\star(x + \d) } }{ \mathcal{D}_{\kappa^\d} \alpha_\star(x - \d) } }{x}&	\\
	\nonumber&= \inty{\field{R}}{}{ \ip{ \overline{ \alpha_\star(x + \d) } }{ ( \sigma_3 i \de_x + \sigma_1 \kappa^\d(x) ) \alpha_\star(x - \d)} }{x}. &	
\end{align*}
Using \eqref{eq:D_acting}, we have that:  
\begin{align*}
	\nonumber&= \gamma^2 \inty{\field{R}}{}{ \ip{ \begin{pmatrix} 1 \\ - i \end{pmatrix} }{ \begin{pmatrix} i \\ 1 \end{pmatrix}} e^{- \inty{0}{x + \d}{\kappa(y)}{y}}  ( \de_x + \kappa^\d(x) ) e^{- \inty{0}{x - \d}{\kappa(y)}{y}} }{x} &	\\
	\nonumber&= 2 i \gamma^2 \inty{\field{R}}{}{ e^{- \inty{0}{x + \d}{\kappa(y)}{y}}  ( \de_x + \kappa^\d(x) ) e^{- \inty{0}{x - \d}{\kappa(y)}{y}} }{x}. &	
\end{align*}
Since $\kappa^\d(x) = \kappa(x - \d)$ for $x \geq 0$:
\begin{equation*}
	\nonumber= 2 i \gamma^2 \inty{-\infty}{0}{ e^{- \inty{0}{x + \d}{\kappa(y)}{y}}  ( \de_x + \kappa^\d(x) ) e^{- \inty{0}{x - \d}{\kappa(y)}{y}} }{x} .
\end{equation*}
Integrating by parts:
\begin{equation*}
\begin{split}
	=& - 2 i \gamma^2 \inty{-\infty}{0}{ \left[ ( \de_x - \kappa^\d(x) ) e^{- \inty{0}{x + \d}{\kappa(y)}{y}} \right] e^{- \inty{0}{x - \d}{\kappa(y)}{y}} }{x} \\
	&+ 2 i \gamma^2 e^{- \inty{0}{ \d}{\kappa(y)}{y}} e^{- \inty{0}{- \d}{\kappa(y)}{y}}.
\end{split}
\end{equation*}
Finally, since $\kappa^\d(x) = - \kappa(x + \d)$ for $x \leq 0$:
\begin{equation*}
	= 2 i \gamma^2 e^{- \inty{0}{ \d}{\kappa(y)}{y}} e^{- \inty{0}{- \d}{\kappa(y)}{y}} = 2 i \gamma^2 e^{- 2 \inty{0}{\d}{\kappa(y)}{y}}, 
\end{equation*}
where the last equality uses the fact that $\kappa$ is odd: $\kappa(-x) = - \kappa(x)$. 

\subsubsection{Proof of \eqref{eq:bound}} \label{sec:4}
Estimate \eqref{eq:bound} follows immediately from Proposition \ref{prop:app_zm} ($\| \mathcal{D}_{\kappa^\d} \alpha^\+_\star \|_\mathcal{H}$ and $\| \mathcal{D}_{\kappa^\d} \alpha^\+_\star \|_\mathcal{H}$ are of $O(e^{- 2 \kappa_\infty \d})$), boundedness of the resolvent operator $(\mathcal{D}_{\kappa^\d} - E)^{-1} \Prl$ in $\mathcal{H}$ for $|E| \leq K$, and the Cauchy-Schwarz inequality. 


\section{Proof of Proposition \ref{prop:bounds_on_terms_in_M_again}} \label{sec:proof_3dw_lemma}
 Using self-adjointness of $\mathcal{D}_{\kappa^\d}$ and conjugate symmetry of $\ip{\cdot}{\cdot}_\mathcal{H}$, Proposition \ref{prop:bounds_on_terms_in_M_again} follows from Lemma \ref{not-orthog} 
 and the following assertions:
\begin{equation} \label{eq:bound_2_again}
\begin{split}
	&\ip{\alpha_\star^\-}{\mathcal{D}_{\kappa^\d} \alpha_\star^\-}_\mathcal{H} = \ip{\alpha_\star^0}{\mathcal{D}_{\kappa^\d} \alpha_\star^0}_\mathcal{H} = \ip{\alpha_\star^\+}{\mathcal{D}_{\kappa^\d} \alpha_\star^\+}_\mathcal{H} = 0	\\
	&\ip{\alpha_\star^\-}{\mathcal{D}_{\kappa^\d} \alpha_\star^\+}_\mathcal{H} = 0,
\end{split}
\end{equation}
\begin{equation} \label{eq:bound_3}
\begin{split}
	&\ip{\alpha_\star^\-}{\mathcal{D}_{\kappa^\d} \alpha_\star^0}_\mathcal{H} = 2 i \gamma^2 e^{ - 2 \inty{0}{\d}{ \kappa(y) }{y} } + O(e^{- 4 \kappa_\infty \d})	\\
	&\ip{\alpha_\star^0}{\mathcal{D}_{\kappa^\d} \alpha_\star^\+}_\mathcal{H} = 2 i \gamma^2 e^{ - 2 \inty{0}{\d}{ \kappa(y) }{y} } + O(e^{- 4 \kappa_\infty \d})	\\
\end{split}
\end{equation}
\begin{equation} \label{eq:bound_again}
\begin{split}
	\Big|\ \ip{ \Prl \mathcal{D}_{\kappa^\d} \alpha_\star^i }{ (\mathcal{D}_{\kappa^\d} - E)^{-1} \Prl \mathcal{D}_{\kappa^\d} \alpha_\star^j }_\mathcal{H}\ \Big|\ \leq \ C e^{- 4 \kappa_\infty \d}, \text{ for } i, j \in \{ \-, 0, \+ \}.  
\end{split}
\end{equation}
The proofs of assertions \eqref{eq:bound_2_again} and \eqref{eq:bound_again} parallel closely the discussions given above in sections \ref{sec:proofs_2} and \ref{sec:4} respectively and hence are omitted. The proof of assertion \eqref{eq:bound_3} is given below:
\subsubsection{Proof of assertion \eqref{eq:bound_3} } 
\begin{align}
\nonumber &\ip{\alpha_\star^\-}{\mathcal{D}_{\kappa^\d} \alpha_\star^0}_\mathcal{H} = \inty{\field{R}}{}{ \ip{\alpha_\star^\-(x)}{\mathcal{D}_{\kappa^\d} \alpha_\star^0(x)} }{x} &	\\
\nonumber &= \inty{\field{R}}{}{ \ip{ \alpha_\star(x + 2 \d) }{ (\sigma_3 i \de_x + \sigma_1 \kappa^\d(x) ) \overline{ \alpha_\star(x) } } }{x}. & \text{(by definition: \eqref{eq:shifted_zero_modes})}
\end{align}
Since $(\sigma_3 i \de_x + \sigma_1 \kappa^\d(x)) \begin{pmatrix} 1 \\ - i \end{pmatrix} = \begin{pmatrix} i \\ -1 \end{pmatrix} (\de_x - \kappa^\d)$, we have that:
\begin{equation}
\begin{split}
	&= \gamma^2 \ip{ \begin{pmatrix} 1 \\ i \end{pmatrix} }{ \begin{pmatrix} i \\ -1 \end{pmatrix} } \inty{\field{R}}{}{ e^{- \inty{0}{x + 2 \d}{ \kappa(y) }{y} } (\de_x - \kappa^\d) e^{- \inty{0}{x}{ \kappa(y) }{y} } }{x}.
\end{split}
\end{equation}
Using the definition of $\kappa^\d(x)$ \eqref{eq:three_dw_kappa}, we have that:
\begin{equation} \label{eq:terms}
\begin{split}
	&= 2 i \gamma^2 \left( \inty{\delta}{\infty}{ e^{- \inty{0}{x + 2 \d}{ \kappa(y) }{y} } (\de_x - \kappa(x - 2 \d) e^{- \inty{0}{x}{ \kappa(y) }{y} } }{x} \right. \\
	&\left. \hphantom{=} + \inty{- \infty}{-\delta}{ e^{ - \inty{0}{x + 2 \d}{\kappa(y)}{y} } (\de_x - \kappa(x + 2 \d)) e^{- \inty{0}{x}{ \kappa(y) }{y} } }{x} \right).
\end{split}
\end{equation}
We now show that the first term in \eqref{eq:terms} may be bounded by $C e^{- 4 \kappa_\infty \d}$. The bounds \eqref{eq:alpha_for_x_large} imply that:
\begin{equation} \label{eq:t_1}
\begin{split}
	&2 i \gamma^2 \inty{\delta}{\infty}{ e^{- \inty{0}{x + 2 \d}{ \kappa(y) }{y} } (\de_x - \kappa(x - 2 \d) e^{- \inty{0}{x}{ \kappa(y) }{y} } }{x} \\
	&\leq C \inty{\delta}{\infty}{ e^{- \kappa_\infty | x + 2 \d |} e^{- \kappa_\infty |x| } }{x} 	\\
	&\leq C e^{- 4 \kappa_\infty \d}. 
\end{split}
\end{equation}
The second term in \eqref{eq:terms} may be evaluated by integrating by parts: 
\begin{equation} \label{eq:t_2}
\begin{split}
	&2 i \gamma^2 \inty{-\infty}{-\d}{ e^{ - \inty{0}{x + 2 \d}{\kappa(y)}{y} } (\de_x - \kappa(x + 2 \d)) e^{- \inty{0}{x}{ \kappa(y) }{y} } }{x} 	\\
	&= 2 i \gamma^2 e^{ - \inty{0}{ \d }{ \kappa(y) }{y} } e^{- \inty{0}{- \d}{ \kappa(y) }{y} } 	\\
	&\phantom{=} - 2 i \gamma^2 \inty{-\infty}{-\d}{ (\de_x + \kappa(x + 2 \d)) e^{ - \inty{0}{x + 2 \d}{\kappa(y)}{y} }  e^{- \inty{0}{x}{ \kappa(y) }{y} } }{x} \\
	&= 2 i \gamma^2 e^{ - \inty{0}{ \d }{ \kappa(y) }{y} } e^{- \inty{0}{- \d}{ \kappa(y) }{y} } = 2 i \gamma^2 e^{ - 2 \inty{0}{\d}{ \kappa(y) }{y} },
\end{split}
\end{equation}
where the last equality follows from $\kappa$ being odd: $\kappa(-x) = - \kappa(x)$. The first assertion of \eqref{eq:bound_3} then follows from adding \eqref{eq:t_1} and \eqref{eq:t_2}. As for the second assertion, we have:
\begin{align}
\nonumber &\ip{\alpha_\star^0}{\mathcal{D}_{\kappa^\d} \alpha_\star^\+}_\mathcal{H} = \inty{\field{R}}{}{ \ip{\alpha_\star^0(x)}{\mathcal{D}_{\kappa^\d} \alpha_\star^\+(x)} }{x} &	\\
\nonumber &= \inty{\field{R}}{}{ \ip{ \overline{ \alpha_\star(x) } }{ (\sigma_3 i \de_x + \sigma_1 \kappa^\d(x) ) \alpha_\star(x - 2\d) } }{x}. & \text{(by definition: \eqref{eq:shifted_zero_modes})}
\end{align}
Using \eqref{eq:D_acting}, we have that:
\begin{equation}
\begin{split}
	&= \gamma^2 \ip{ \begin{pmatrix} 1 \\ - i \end{pmatrix} }{ \begin{pmatrix} i \\ 1 \end{pmatrix} } \inty{\field{R}}{}{ e^{- \inty{0}{x}{ \kappa(y) }{y} } (\de_x + \kappa^\d) e^{- \inty{0}{x - 2 \d}{ \kappa(y) }{y} } }{x}.
\end{split}
\end{equation}
Using the definition of $\kappa^\d(x)$ \eqref{eq:three_dw_kappa}, we have that:
\begin{equation} \label{eq:terms_again}
\begin{split}
	= &2 i \gamma^2 \left( \inty{- \delta}{\d}{ e^{- \inty{0}{x}{ \kappa(y) }{y} } (\de_x - \kappa(x)) e^{- \inty{0}{x - 2\d}{ \kappa(y) }{y} } }{x} \right. \\
	&\left. + \inty{- \infty}{-\delta}{ e^{ - \inty{0}{x}{\kappa(y)}{y} } (\de_x + \kappa(x + 2 \d)) e^{- \inty{0}{x - 2\d}{ \kappa(y) }{y} } }{x} \right).
\end{split}
\end{equation}
Using the bounds \eqref{eq:alpha_for_x_large}, we see that the second term in \eqref{eq:terms_again} may be bounded by $C e^{- 4 \kappa_\infty \d}$: 
\begin{equation} \label{eq:t_1_again}
	\inty{- \infty}{-\delta}{ e^{ - \inty{0}{x}{\kappa(y)}{y} } (\de_x + \kappa(x + 2 \d)) e^{- \inty{0}{x - 2\d}{ \kappa(y) }{y} } }{x} \leq C e^{- 4 \kappa_\infty \d}. 
\end{equation}
Integrating by parts in the first term in \eqref{eq:terms} then gives: 
\begin{equation} 
\begin{split}
	&2 i \gamma^2 \inty{- \d}{\d}{ e^{- \inty{0}{x}{ \kappa(y) }{y} } (\de_x - \kappa(x)) e^{- \inty{0}{x - 2\d}{ \kappa(y) }{y} } }{x}  \\
	&= 2 i \gamma^2 \left[ e^{- \inty{0}{x}{ \kappa(y) }{y} } e^{- \inty{0}{x - 2\d}{\kappa(y)}{y}} \right]_{-\d}^\d				\\ 
	&\phantom{=} - 2 i \gamma^2 \inty{- \d}{\d}{ (\de_x + \kappa(x)) e^{- \inty{0}{x}{ \kappa(y) }{y} }  e^{- \inty{0}{x - 2\d}{ \kappa(y) }{y} } }{x}  \\
	&= 2 i \gamma^2 \left[ e^{- \inty{0}{x}{ \kappa(y) }{y} } e^{- \inty{0}{x - 2\d}{\kappa(y)}{y}} \right]_{-\d}^\d				\\ 
	&= 2 i \gamma^2 e^{- \inty{0}{\delta}{ \kappa(y) }{y} } e^{- \inty{0}{- \delta}{\kappa(y)}{y}} - 2 i \gamma^2 e^{- \inty{0}{- \delta}{\kappa(y)}{y} } e^{- \inty{0}{- 3 \delta}{\kappa(y)}{y}} 
\end{split}
\end{equation}
which implies that:
\begin{equation} \label{eq:t_2_again}
\begin{split}
	&2 i \gamma^2 \inty{- \d}{\d}{ e^{- \inty{0}{x}{ \kappa(y) }{y} } (\de_x - \kappa(x)) e^{- \inty{0}{x - 2\d}{ \kappa(y) }{y} } }{x} 	\\
	&= 2 i \gamma^2 e^{- \inty{0}{\delta}{ \kappa(y) }{y} } e^{- \inty{0}{- \delta}{\kappa(y)}{y}} + O(e^{- 4 \kappa_\infty \d})	\\
	&= 2 i \gamma^2 e^{- 2 \inty{0}{\delta}{ \kappa(y) }{y} } + O(e^{- 4 \kappa_\infty \d}) 
\end{split}
\end{equation}
where the last equality holds because $\kappa$ is odd. Adding \eqref{eq:t_1_again} and \eqref{eq:t_2_again} implies the second part of assertion \eqref{eq:bound_3}.

\printbibliography

\end{document}